\documentclass{amsart}
\usepackage{amsmath,amssymb,graphicx,latexsym}
\usepackage[usenames,dvipsnames]{color}
\usepackage{hyperref}

\newtheorem{theorem}{Theorem}
\newtheorem{conjecture}[theorem]{Conjecture}
\newtheorem{corollary}[theorem]{Corollary}

\theoremstyle{definition}
\newtheorem{example}[theorem]{Example}
\newtheorem*{definition*}{Definition}
\newtheorem*{notation*}{Notation}

\DeclareMathOperator{\adj}{adj}
\DeclareMathOperator{\pos}{pos}
\DeclareMathOperator{\Sink}{Sink}

\newcommand{\Q}{\mathbb{Q}}
\newcommand{\Z}{\mathbb{Z}}

\newcommand{\colonequal}{\mathrel{\mathop:}=}
\newcommand{\tr}{^\textnormal{\textsf{T}}}

\newcommand{\size}[1]{\lvert{#1}\rvert}

\newcommand{\subsetmatrix}[1]{
	\begin{smallmatrix}
		#1
	\end{smallmatrix}
}

\newcommand{\seq}[1]{\cite[\href{http://oeis.org/#1}{#1}]{OEIS}}

\begin{document}

\title[The entries of the Sinkhorn limit of an $m \times n$ matrix]{The entries of the Sinkhorn \\ limit of an $m \times n$ matrix}
\author{Eric Rowland}
\address{
	Department of Mathematics \\
	Hofstra University \\
	Hempstead, NY \\
	USA
}
\author{Jason Wu}
\address{
	Half Hollow Hills High School East \\
	Dix Hills, NY \\
	USA
}
\curraddr{Department of Mathematics \\
	Cornell University \\
	Ithaca, NY \\
	USA
}
\date{May 25, 2025}

\begin{abstract}
We use a variety of computational tools to obtain a degree-$\binom{m + n - 2}{m - 1}$ polynomial equation conjecturally satisfied by the top-left entry of the Sinkhorn limit of a positive $m \times n$ matrix.
The degree of this equation has a combinatorial interpretation as the number of minors of an $(m - 1) \times (n - 1)$ matrix, and the coefficients involve a determinant formula that reflects new combinatorial structure on sets of minor specifications.
The tools we use include Gr\"obner bases, which produce equations for small matrices; the PSLQ algorithm, which produces equations for larger matrices as part of an interpolation effort that required $1.5$~years of CPU time; and ChatGPT o3-mini-high, which identified the signs of the off-diagonal entries in the determinant formula.
\end{abstract}

\maketitle

\section[Introduction]{Introduction\protect\footnotemark}\label{Introduction}

\footnotetext{For a video introduction to this paper, see \url{https://youtu.be/-uIwboK4nwE}.}

In a 1964 paper, Sinkhorn~\cite{Sinkhorn} considered the following iterative scaling process.
Let $A$ be a square matrix with positive entries.
Scale the rows so that each row sum is $1$.
Then scale the columns so that each column sum is $1$; generically, this changes the row sums.
To restore the row sums $1$, scale the rows again, then scale the columns, and so on.
Sinkhorn showed that the sequence of matrices obtained through this process converges to a matrix whose row and column sums are $1$ (in other words, a \emph{doubly stochastic} matrix).
We call this matrix the \emph{Sinkhorn limit} of $A$ and denote it $\Sink(A)$.
Sinkhorn also showed that $\Sink(A)$ is the unique doubly stochastic matrix $S$ with the same dimensions as $A$ such that $S = D_1 A D_2$ for some diagonal matrices $D_1$ and $D_2$ with positive diagonal entries.
Here $D_1$ can be taken to be the product of the row-scaling matrices from the iterative scaling process and $D_2$ the product of the column-scaling matrices.

The literature on questions related to the iterative scaling process is large;
Idel's extensive survey~\cite{Idel} covers results as of 2016.
Mathematical applications include preconditioning linear systems to improve numerical stability, approximating the permanent of a matrix, and determining whether a graph has a perfect matching.
In a number of other areas, iterative scaling was discovered independently~\cite{Kruithof, Deming--Stephan, Brown}, and it is used in machine learning to efficiently compute optimal transport distances~\cite{Cuturi}.
As a result of its ubiquity and importance, many authors have been interested in fast algorithms for approximating Sinkhorn limits numerically~\cite{Franklin--Lorenz, Kalantari--Khachiyan 1, Kalantari--Khachiyan 2, Linial--Samorodnitsky--Wigderson, Kalantari--Lari--Ricca--Simeone, Allen-Zhu--Li--Oliveira--Wigderson, Cohen--Madry--Tsipras--Vladu}.

However, until recently, nothing was known about the exact values of the entries of Sinkhorn limits.
For a $2 \times 2$ matrix
\[
	A =
	\begin{bmatrix}
		a & b \\
		c & d
	\end{bmatrix}
\]
with positive entries, Nathanson~\cite[Theorem~3]{Nathanson 2x2} showed that
\[
	\Sink(A) =
	\frac{1}{\sqrt{a d} + \sqrt{b c}} \begin{bmatrix}
		\sqrt{a d} & \sqrt{b c} \\
		\sqrt{b c} & \sqrt{a d}
	\end{bmatrix}.
\]
In particular, the top-left entry $x$ of $\Sink(A)$ satisfies
\begin{equation}\label{2 by 2 polynomial}
	(a d - b c) x^2 - 2 a d x + a d = 0.
\end{equation}

For $3 \times 3$ matrices, analogous descriptions of the entries of $\Sink(A)$ were only known in special cases.
In the case that $A$ is a symmetric $3 \times 3$ matrix containing exactly $2$ distinct entries, a formula for $\Sink(A)$ was obtained by Nathanson~\cite{Nathanson 3x3}.
For a symmetric $3 \times 3$ matrix, Ekhad and Zeilberger~\cite{Ekhad--Zeilberger} used Gr\"obner bases to compute, for each entry $x$ of $\Sink(A)$, a degree-$4$ polynomial of which $x$ is a root.

For a general $3 \times 3$ matrix $A$, Chen and Varghese~\cite{Chen--Varghese} used numeric experiments to conjecture that the entries of $\Sink(A)$ generically have degree $6$ over the field generated by the entries of $A$.

\begin{example}\label{3 by 3 example}
Applying the iterative scaling process to
\[
	A =
	\begin{bmatrix}
		3 & 9 & 1 \\
		3 & 2 & 9 \\
		5 & 3 & 4
	\end{bmatrix}
\]
gives the approximation
\[
	\Sink(A) \approx
	\begin{bmatrix}
		0.27667 & 0.64804 & 0.07527 \\
		0.25194 & 0.13113 & 0.61692 \\
		0.47138 & 0.22081 & 0.30780
	\end{bmatrix}.
\]
Let $x$ be the top-left entry of $\Sink(A)$.
The PSLQ integer relation algorithm~\cite{Ferguson--Bailey} can be used to recognize an algebraic number, given a sufficiently high-precision approximation.
For the approximation
\[
	x \approx 0.2766771162103280503525099931476512576251224460918253185145079454
\]
with target degree $6$, PSLQ produces the guess
\begin{equation}\label{3 by 3 example matrix polynomial}
	374752 x^6 - 220388 x^5 - 844359 x^4 - 125796 x^3 + 210897 x^2 + 14346 x - 12312 = 0.
\end{equation}
This guess remains stable when we increase the precision or the target degree.
\end{example}

In Section~\ref{3 by 3}, we prove the conjecture of Chen and Varghese by carrying out a Gr\"obner basis computation to obtain an explicit polynomial equation satisfied by an entry of $\Sink(A)$ for a general $3 \times 3$ matrix~$A$ with positive entries.
In particular, we obtain formulas for the coefficients in Equation~\eqref{3 by 3 example matrix polynomial}.
We show that each coefficient in this polynomial can be written as a linear combination of products of minors of $A$ (that is, determinants of square submatrices); this substantially reduces the amount of information required to specify the coefficients.

In Section~\ref{Square matrices}, we use the form of the equation for $3 \times 3$ matrices to infer the form of the equation for $n \times n$ matrices $A$.
In particular, we conjecture that the entries of $\Sink(A)$ generically have degree $\binom{2 n - 2}{n - 1}$.
To obtain the equation for a $4 \times 4$ matrix, the Gr\"obner basis computation is infeasible, so instead we use PSLQ to recognize entries of Sinkhorn limits for enough integer matrices to solve for each coefficient in the equation.
We also obtain several coefficients in the equation for $5 \times 5$ matrices with this method.
However, this does not provide enough data to identify the coefficients for $n \times n$ matrices in general.

Therefore, in Section~\ref{Rectangular matrices and combinatorial structure}, we generalize to rectangular matrices by defining the Sinkhorn limit of an $m \times n$ matrix~$A$ to be the matrix obtained by iteratively scaling so that each row sum is $1$ and each column sum is $\frac{m}{n}$.
Again we use PSLQ, and we solve for coefficients in equations for matrices of various dimensions.
We then interpolate formulas for these coefficients as functions of $m$ and $n$.
By identifying combinatorial structure in these formulas, we obtain Conjecture~\ref{general polynomial} below.
To state it, we introduce notation for the relevant products of minors; see Section~\ref{3 by 3} for examples.

\begin{notation*}
Define
\[
	D(m, n) = \{(R, C) : \text{$R \subseteq \{2, 3, \dots, m\}$ and $C \subseteq \{2, 3, \dots, n\}$ and $\size{R} = \size{C}$}\}.
\]
The set $D(m, n)$ consists of the specifications of all minors of an $m \times n$ matrix~$A$ that do not involve the first row or first column.
We have $\size{D(m, n)} = \sum_{k = 0}^{\min(m, n) - 1} \binom{m - 1}{k} \binom{n - 1}{k} = \binom{m + n - 2}{m - 1}$.
Define $A_{R, C}$ to be the submatrix of $A$ obtained by extracting the rows indexed by $R$ and the columns indexed by $C$.
For each $(R, C) \in D(m, n)$, define
\begin{align*}
	\Delta\!\left(\!\begin{smallmatrix}
		R\vphantom{\{} \\
		C\vphantom{\{}
	\end{smallmatrix}\!\right)
	&= \det A_{\{1\} \cup R, \{1\} \cup C} \\
	\Gamma\!\left(\!\begin{smallmatrix}
		R\vphantom{\{} \\
		C\vphantom{\{}
	\end{smallmatrix}\!\right)
	&= a_{11} \det A_{R, C}.
\end{align*}
The minor
$
	\Delta\!\left(\!\begin{smallmatrix}
		R\vphantom{\{} \\
		C\vphantom{\{}
	\end{smallmatrix}\!\right)
$
involves the top-left entry $a_{11}$, and
$
	\Gamma\!\left(\!\begin{smallmatrix}
		R\vphantom{\{} \\
		C\vphantom{\{}
	\end{smallmatrix}\!\right)
$
is the product of $a_{11}$ and a minor that does not involve the first row or first column.
This notation does not reflect the dependence on $A$, but the matrix will be clear from context.
Each subset $S \subseteq D(m, n)$ specifies a monomial in the expressions
$
	\Delta\!\left(\!\begin{smallmatrix}
		R\vphantom{\{} \\
		C\vphantom{\{}
	\end{smallmatrix}\!\right)
$
and
$
	\Gamma\!\left(\!\begin{smallmatrix}
		R\vphantom{\{} \\
		C\vphantom{\{}
	\end{smallmatrix}\!\right)
$, namely
\[
	M(S)
	= \prod_{(R, C) \in S}
	\Delta\!\left(\!\begin{smallmatrix}
		R\vphantom{\{} \\
		C\vphantom{\{}
	\end{smallmatrix}\!\right)
	\cdot
	\prod_{(R, C) \in D(m, n) \setminus S}
	\Gamma\!\left(\!\begin{smallmatrix}
		R\vphantom{\{} \\
		C\vphantom{\{}
	\end{smallmatrix}\!\right).
\]
Each element of $D(m, n)$ contributes to the monomial $M(S)$ as the argument of either $\Delta$ or $\Gamma$; indeed we specify each monomial by the elements that appear as arguments of $\Delta$.
\end{notation*}

The coefficients of these monomials in Conjecture~\ref{general polynomial} are given by a determinant formula involving an adjacency-like (or possibly Laplacian-like) matrix.
We define this matrix next, and we discuss the combinatorial structure behind it in Section~\ref{Rectangular matrices and combinatorial structure}.

\begin{notation*}
If $X$ is a set of positive integers and $s \in X$, define $\pos_s(X)$ to be the position, indexing from $1$, in which $s$ appears when the elements of $X$ are listed in order.
For example, $\pos_5(\{2, 4, 5\}) = 3$.
\end{notation*}

\begin{notation*}
For each subset $S = \{(R_1, C_1), (R_2, C_2), \dots, (R_k, C_k)\} \subseteq D(m, n)$, let $\adj_S(m, n)$ be the $k \times k$ matrix defined as follows.
The $(i, i)$ diagonal entry of $\adj_S(m, n)$ is $\size{R_i} (m + n) - m n$.
For $i \neq j$, define the $4$-tuple
\[
	\tau_{ij} = \left(
		R_i \setminus R_j, \quad
		R_j \setminus R_i, \quad
		C_i \setminus C_j, \quad
		C_j \setminus C_i
	\right).
\]
The $(i, j)$ off-diagonal entry of $\adj_S(m, n)$ is
\[
	\begin{cases}
		(-1)^{\pos_s(R_j) + \pos_t(C_j)} m		& \text{if $\tau_{ij} = (\{\}, \{s\}, \{\}, \{t\})$ for some $s, t$} \\
		(-1)^{\pos_s(R_i) + \pos_t(C_i) + 1} n	& \text{if $\tau_{ij} = (\{s\}, \{\}, \{t\}, \{\})$ for some $s, t$} \\
		(-1)^{\pos_s(C_i) + \pos_t(C_j)} m		& \text{if $\tau_{ij} = (\{\}, \{\}, \{s\}, \{t\})$ for some $s, t$} \\
		(-1)^{\pos_s(R_i) + \pos_t(R_j)} n		& \text{if $\tau_{ij} = (\{s\}, \{t\}, \{\}, \{\})$ for some $s, t$} \\
		0								& \text{otherwise.}
	\end{cases}
\]
In defining $\adj_S(m, n)$, we have chosen an order on the elements of $S$.
However, $\det \adj_S(m, n)$ does not depend on this order, since swapping $(R_i, C_i)$ and $(R_j, C_j)$ in the order on $S$ simply has the effect of swapping rows $i$ and $j$ and columns $i$ and $j$ in $\adj_S(m, n)$.
Also note that, if we fix an order on $D(m, n)$ and use that order for each $S \subseteq D(m, n)$, then each $\adj_S(m, n)$ is a submatrix of $\adj_{D(m, n)}(m, n)$; this is computationally useful since it is faster to compute $\adj_{D(m, n)}(m, n)$ once and extract submatrices than to compute all $\adj_S(m, n)$ from the definition.
\end{notation*}

\begin{conjecture}\label{general polynomial}
Let $m \geq 1$ and $n \geq 1$.
For every $m \times n$ matrix~$A$ with positive entries, the top-left entry $x$ of $\Sink(A)$ satisfies
\[
	\sum_{S \subseteq D(m, n)} \det\!\left(\tfrac{1}{m} \adj_S(m, n)\right) M(S) x^{\size{S}} = 0.
\]
In particular, $x$ is algebraic over the field generated by the entries of $A$, with degree at most $\binom{m + n - 2}{m - 1}$.
\end{conjecture}

The factor $\frac{1}{m}$ in Conjecture~\ref{general polynomial} arises from the asymmetry between rows and columns in our definition of $\Sink(A)$.

We conclude in Section~\ref{Open questions} with several open questions.
In particular, we mention that a more general iterative scaling process was introduced in 1937 by Kruithof in the context of predicting telephone traffic~\cite[Appendix~3d]{Kruithof}.
We conjecture that the entries of this more general limit also have degree at most $\binom{m + n - 2}{m - 1}$.

Our Mathematica package \textsc{SinkhornPolynomials}~\cite{SinkhornPolynomials} uses the results of this paper to compute Sinkhorn limits rigorously for $3 \times 3$ matrices and conjecturally for larger matrices.

\section{The Sinkhorn limit of a $3 \times 3$ matrix}\label{3 by 3}

We refer to a matrix with positive entries as a \emph{positive matrix}.
In this section, we determine $\Sink(A)$ for a general positive $3 \times 3$ matrix
\begin{equation}\label{3 by 3 matrix}
	A =
	\begin{bmatrix}
		a_{11} & a_{12} & a_{13} \\
		a_{21} & a_{22} & a_{23} \\
		a_{31} & a_{32} & a_{33}
	\end{bmatrix},
\end{equation}
first in Theorem~\ref{3 by 3 polynomial - large} as an explicit function of the entries of $A$ and then in Theorem~\ref{3 by 3 polynomial} in a form that shows more structure.
We also introduce notation that will be used throughout the rest of the paper.

It suffices to describe the top-left entry of $\Sink(A)$.
This is because the iterative scaling process isn't sensitive to the order of the rows or the order of the columns.
If $P_1$ is the permutation matrix swapping rows $1$ and $i$ and $P_2$ is the permutation matrix swapping columns $1$ and $j$, then $P_1 \Sink(A) P_2 = \Sink(P_1 A P_2)$.
In particular, the $(i, j)$ entry of $\Sink(A)$ is equal to the $(1, 1)$ entry of $\Sink(P_1 A P_2)$.

To compute a polynomial equation satisfied by the top-left entry of $\Sink(A)$, we set up three matrices
\[
	S =
	\begin{bmatrix}
		s_{11} & s_{12} & s_{13} \\
		s_{21} & s_{22} & s_{23} \\
		s_{31} & s_{32} & s_{33}
	\end{bmatrix},
	\quad
	R =
	\begin{bmatrix}
		r_1 & 0 & 0 \\
		0 & r_2 & 0 \\
		0 & 0 & r_3
	\end{bmatrix},
	\quad
	C =
	\begin{bmatrix}
		c_1 & 0 & 0 \\
		0 & c_2 & 0 \\
		0 & 0 & c_3
	\end{bmatrix}.
\]
The matrix equation $S = R A C$ gives the $9$ equations
\begin{align*}
	s_{11} &= r_1 a_{11} c_1		&	s_{12} &= r_1 a_{12} c_2		&	s_{13} &= r_1 a_{13} c_3 \\
	s_{21} &= r_2 a_{21} c_1		&	s_{22} &= r_2 a_{22} c_2		&	s_{23} &= r_2 a_{23} c_3 \\
	s_{31} &= r_3 a_{31} c_1		&	s_{32} &= r_3 a_{32} c_2		&	s_{33} &= r_3 a_{33} c_3,
\end{align*}
and we obtain $6$ equations from the requirement that $S$ is doubly stochastic:
\begin{align*}
	s_{11} + s_{12} + s_{13} &= 1	&	s_{11} + s_{21} + s_{31} &= 1 \\
	s_{21} + s_{22} + s_{23} &= 1	&	s_{12} + s_{22} + s_{32} &= 1 \\
	s_{31} + s_{32} + s_{33} &= 1	&	s_{13} + s_{23} + s_{33} &= 1.
\end{align*}
We would like to eliminate the $14$ variables $s_{12}, s_{13}, \dots, s_{33}, r_1, r_2, r_3, c_1, c_2, c_3$ from this system of $15$ polynomial equations, resulting in a single equation in the variables $s_{11}, a_{11}, a_{12}, \dots, a_{33}$.
In principle, this can be done by computing a suitable Gr\"obner basis.
In practice, the runtime is significantly affected by the algorithm used.
Mathematica's \texttt{GroebnerBasis} function~\cite{GroebnerBasis} with certain settings\footnote{Namely, \texttt{Method -> "Buchberger", MonomialOrder -> EliminationOrder, Sort -> True}.} computes a single polynomial in a couple seconds, whereas with other settings the computation does not finish after several hours.
The output gives the following result.

\begin{theorem}\label{3 by 3 polynomial - large}
Let $A$ be a positive $3 \times 3$ matrix.
The top-left entry $x$ of $\Sink(A)$ satisfies $b_6 x^6 + \dots + b_1 x + b_0 = 0$, where the coefficients $b_k$ appear in Table~\ref{3 by 3 coefficients} in factored form.
\end{theorem}

\begin{table}
\tiny
\begin{align*}
	b_6 = {} &
		\left(a_{11} a_{22} - a_{12} a_{21}\right)
		\left(a_{11} a_{23} - a_{13} a_{21}\right)
		\left(a_{11} a_{32} - a_{12} a_{31}\right)
		\left(a_{11} a_{33} - a_{13} a_{31}\right) \\
		& \cdot \left(a_{11} a_{22} a_{33} - a_{11} a_{23} a_{32} - a_{12} a_{21} a_{33} + a_{12} a_{23} a_{31} + a_{13} a_{21} a_{32} - a_{13} a_{22} a_{31}\right)
	\\
	b_5 = {} &
		{-}6 a_{11}^5 a_{22}^2 a_{23}^{\vphantom{1}} a_{32}^{\vphantom{1}} a_{33}^2
		+ 6 a_{11}^5 a_{22}^{\vphantom{1}} a_{23}^2 a_{32}^2 a_{33}^{\vphantom{1}}
		+ 8 a_{11}^4 a_{12}^{\vphantom{1}} a_{21}^{\vphantom{1}} a_{22}^{\vphantom{1}} a_{23}^{\vphantom{1}} a_{32}^{\vphantom{1}} a_{33}^2 \\
		& - 5 a_{11}^4 a_{12}^{\vphantom{1}} a_{21}^{\vphantom{1}} a_{23}^2 a_{32}^2 a_{33}^{\vphantom{1}}
		+ 5 a_{11}^4 a_{12}^{\vphantom{1}} a_{22}^2 a_{23}^{\vphantom{1}} a_{31}^{\vphantom{1}} a_{33}^2
		- 8 a_{11}^4 a_{12}^{\vphantom{1}} a_{22}^{\vphantom{1}} a_{23}^2 a_{31}^{\vphantom{1}} a_{32}^{\vphantom{1}} a_{33}^{\vphantom{1}} \\
		& + 5 a_{11}^4 a_{13}^{\vphantom{1}} a_{21}^{\vphantom{1}} a_{22}^2 a_{32}^{\vphantom{1}} a_{33}^2
		- 8 a_{11}^4 a_{13}^{\vphantom{1}} a_{21}^{\vphantom{1}} a_{22}^{\vphantom{1}} a_{23}^{\vphantom{1}} a_{32}^2 a_{33}^{\vphantom{1}}
		+ 8 a_{11}^4 a_{13}^{\vphantom{1}} a_{22}^2 a_{23}^{\vphantom{1}} a_{31}^{\vphantom{1}} a_{32}^{\vphantom{1}} a_{33}^{\vphantom{1}} \\
		& - 5 a_{11}^4 a_{13}^{\vphantom{1}} a_{22}^{\vphantom{1}} a_{23}^2 a_{31}^{\vphantom{1}} a_{32}^2
		- 2 a_{11}^3 a_{12}^2 a_{21}^2 a_{23}^{\vphantom{1}} a_{32}^{\vphantom{1}} a_{33}^2
		- 6 a_{11}^3 a_{12}^2 a_{21}^{\vphantom{1}} a_{22}^{\vphantom{1}} a_{23}^{\vphantom{1}} a_{31}^{\vphantom{1}} a_{33}^2 \\
		& + 6 a_{11}^3 a_{12}^2 a_{21}^{\vphantom{1}} a_{23}^2 a_{31}^{\vphantom{1}} a_{32}^{\vphantom{1}} a_{33}^{\vphantom{1}}
		+ 2 a_{11}^3 a_{12}^2 a_{22}^{\vphantom{1}} a_{23}^2 a_{31}^2 a_{33}^{\vphantom{1}}
		- 6 a_{11}^3 a_{12}^{\vphantom{1}} a_{13}^{\vphantom{1}} a_{21}^2 a_{22}^{\vphantom{1}} a_{32}^{\vphantom{1}} a_{33}^2 \\
		& + 6 a_{11}^3 a_{12}^{\vphantom{1}} a_{13}^{\vphantom{1}} a_{21}^2 a_{23}^{\vphantom{1}} a_{32}^2 a_{33}^{\vphantom{1}}
		- 4 a_{11}^3 a_{12}^{\vphantom{1}} a_{13}^{\vphantom{1}} a_{21}^{\vphantom{1}} a_{22}^2 a_{31}^{\vphantom{1}} a_{33}^2
		+ 4 a_{11}^3 a_{12}^{\vphantom{1}} a_{13}^{\vphantom{1}} a_{21}^{\vphantom{1}} a_{23}^2 a_{31}^{\vphantom{1}} a_{32}^2 \\
		& - 6 a_{11}^3 a_{12}^{\vphantom{1}} a_{13}^{\vphantom{1}} a_{22}^2 a_{23}^{\vphantom{1}} a_{31}^2 a_{33}^{\vphantom{1}}
		+ 6 a_{11}^3 a_{12}^{\vphantom{1}} a_{13}^{\vphantom{1}} a_{22}^{\vphantom{1}} a_{23}^2 a_{31}^2 a_{32}^{\vphantom{1}}
		+ 2 a_{11}^3 a_{13}^2 a_{21}^2 a_{22}^{\vphantom{1}} a_{32}^2 a_{33}^{\vphantom{1}} \\
		& - 6 a_{11}^3 a_{13}^2 a_{21}^{\vphantom{1}} a_{22}^2 a_{31}^{\vphantom{1}} a_{32}^{\vphantom{1}} a_{33}^{\vphantom{1}}
		+ 6 a_{11}^3 a_{13}^2 a_{21}^{\vphantom{1}} a_{22}^{\vphantom{1}} a_{23}^{\vphantom{1}} a_{31}^{\vphantom{1}} a_{32}^2
		- 2 a_{11}^3 a_{13}^2 a_{22}^2 a_{23}^{\vphantom{1}} a_{31}^2 a_{32}^{\vphantom{1}} \\
		& + a_{11}^2 a_{12}^3 a_{21}^2 a_{23}^{\vphantom{1}} a_{31}^{\vphantom{1}} a_{33}^2
		- a_{11}^2 a_{12}^3 a_{21}^{\vphantom{1}} a_{23}^2 a_{31}^2 a_{33}^{\vphantom{1}}
		+ a_{11}^2 a_{12}^2 a_{13}^{\vphantom{1}} a_{21}^3 a_{32}^{\vphantom{1}} a_{33}^2 \\
		& + 4 a_{11}^2 a_{12}^2 a_{13}^{\vphantom{1}} a_{21}^2 a_{22}^{\vphantom{1}} a_{31}^{\vphantom{1}} a_{33}^2
		- 4 a_{11}^2 a_{12}^2 a_{13}^{\vphantom{1}} a_{21}^2 a_{23}^{\vphantom{1}} a_{31}^{\vphantom{1}} a_{32}^{\vphantom{1}} a_{33}^{\vphantom{1}} \\
		& + 4 a_{11}^2 a_{12}^2 a_{13}^{\vphantom{1}} a_{21}^{\vphantom{1}} a_{22}^{\vphantom{1}} a_{23}^{\vphantom{1}} a_{31}^2 a_{33}^{\vphantom{1}}
		- 4 a_{11}^2 a_{12}^2 a_{13}^{\vphantom{1}} a_{21}^{\vphantom{1}} a_{23}^2 a_{31}^2 a_{32}^{\vphantom{1}}
		- a_{11}^2 a_{12}^2 a_{13}^{\vphantom{1}} a_{22}^{\vphantom{1}} a_{23}^2 a_{31}^3 \\
		& - a_{11}^2 a_{12}^{\vphantom{1}} a_{13}^2 a_{21}^3 a_{32}^2 a_{33}^{\vphantom{1}}
		+ 4 a_{11}^2 a_{12}^{\vphantom{1}} a_{13}^2 a_{21}^2 a_{22}^{\vphantom{1}} a_{31}^{\vphantom{1}} a_{32}^{\vphantom{1}} a_{33}^{\vphantom{1}}
		- 4 a_{11}^2 a_{12}^{\vphantom{1}} a_{13}^2 a_{21}^2 a_{23}^{\vphantom{1}} a_{31}^{\vphantom{1}} a_{32}^2 \\
		& + 4 a_{11}^2 a_{12}^{\vphantom{1}} a_{13}^2 a_{21}^{\vphantom{1}} a_{22}^2 a_{31}^2 a_{33}^{\vphantom{1}}
		- 4 a_{11}^2 a_{12}^{\vphantom{1}} a_{13}^2 a_{21}^{\vphantom{1}} a_{22}^{\vphantom{1}} a_{23}^{\vphantom{1}} a_{31}^2 a_{32}^{\vphantom{1}}
		+ a_{11}^2 a_{12}^{\vphantom{1}} a_{13}^2 a_{22}^2 a_{23}^{\vphantom{1}} a_{31}^3 \\
		& - a_{11}^2 a_{13}^3 a_{21}^2 a_{22}^{\vphantom{1}} a_{31}^{\vphantom{1}} a_{32}^2
		+ a_{11}^2 a_{13}^3 a_{21}^{\vphantom{1}} a_{22}^2 a_{31}^2 a_{32}^{\vphantom{1}}
		- 2 a_{11}^{\vphantom{1}} a_{12}^2 a_{13}^2 a_{21}^2 a_{22}^{\vphantom{1}} a_{31}^2 a_{33}^{\vphantom{1}} \\
		& + 2 a_{11}^{\vphantom{1}} a_{12}^2 a_{13}^2 a_{21}^2 a_{23}^{\vphantom{1}} a_{31}^2 a_{32}^{\vphantom{1}}
		- a_{12}^3 a_{13}^2 a_{21}^3 a_{31}^2 a_{33}^{\vphantom{1}}
		+ a_{12}^3 a_{13}^2 a_{21}^2 a_{23}^{\vphantom{1}} a_{31}^3 \\
		& + a_{12}^2 a_{13}^3 a_{21}^3 a_{31}^2 a_{32}^{\vphantom{1}}
		- a_{12}^2 a_{13}^3 a_{21}^2 a_{22}^{\vphantom{1}} a_{31}^3
	\\
	b_4 = {} & a_{11}^{\vphantom{1}} \big(
		15 a_{11}^4 a_{22}^2 a_{23}^{\vphantom{1}} a_{32}^{\vphantom{1}} a_{33}^2
		- 15 a_{11}^4 a_{22}^{\vphantom{1}} a_{23}^2 a_{32}^2 a_{33}^{\vphantom{1}}
		- 12 a_{11}^3 a_{12}^{\vphantom{1}} a_{21}^{\vphantom{1}} a_{22}^{\vphantom{1}} a_{23}^{\vphantom{1}} a_{32}^{\vphantom{1}} a_{33}^2 \\
		& + 10 a_{11}^3 a_{12}^{\vphantom{1}} a_{21}^{\vphantom{1}} a_{23}^2 a_{32}^2 a_{33}^{\vphantom{1}}
		- 10 a_{11}^3 a_{12}^{\vphantom{1}} a_{22}^2 a_{23}^{\vphantom{1}} a_{31}^{\vphantom{1}} a_{33}^2
		+ 12 a_{11}^3 a_{12}^{\vphantom{1}} a_{22}^{\vphantom{1}} a_{23}^2 a_{31}^{\vphantom{1}} a_{32}^{\vphantom{1}} a_{33}^{\vphantom{1}} \\
		& - 10 a_{11}^3 a_{13}^{\vphantom{1}} a_{21}^{\vphantom{1}} a_{22}^2 a_{32}^{\vphantom{1}} a_{33}^2
		+ 12 a_{11}^3 a_{13}^{\vphantom{1}} a_{21}^{\vphantom{1}} a_{22}^{\vphantom{1}} a_{23}^{\vphantom{1}} a_{32}^2 a_{33}^{\vphantom{1}}
		- 12 a_{11}^3 a_{13}^{\vphantom{1}} a_{22}^2 a_{23}^{\vphantom{1}} a_{31}^{\vphantom{1}} a_{32}^{\vphantom{1}} a_{33}^{\vphantom{1}} \\
		& + 10 a_{11}^3 a_{13}^{\vphantom{1}} a_{22}^{\vphantom{1}} a_{23}^2 a_{31}^{\vphantom{1}} a_{32}^2
		+ a_{11}^2 a_{12}^2 a_{21}^2 a_{23}^{\vphantom{1}} a_{32}^{\vphantom{1}} a_{33}^2
		+ 6 a_{11}^2 a_{12}^2 a_{21}^{\vphantom{1}} a_{22}^{\vphantom{1}} a_{23}^{\vphantom{1}} a_{31}^{\vphantom{1}} a_{33}^2 \\
		& - 6 a_{11}^2 a_{12}^2 a_{21}^{\vphantom{1}} a_{23}^2 a_{31}^{\vphantom{1}} a_{32}^{\vphantom{1}} a_{33}^{\vphantom{1}}
		- a_{11}^2 a_{12}^2 a_{22}^{\vphantom{1}} a_{23}^2 a_{31}^2 a_{33}^{\vphantom{1}}
		+ 6 a_{11}^2 a_{12}^{\vphantom{1}} a_{13}^{\vphantom{1}} a_{21}^2 a_{22}^{\vphantom{1}} a_{32}^{\vphantom{1}} a_{33}^2 \\
		& - 6 a_{11}^2 a_{12}^{\vphantom{1}} a_{13}^{\vphantom{1}} a_{21}^2 a_{23}^{\vphantom{1}} a_{32}^2 a_{33}^{\vphantom{1}}
		+ 6 a_{11}^2 a_{12}^{\vphantom{1}} a_{13}^{\vphantom{1}} a_{21}^{\vphantom{1}} a_{22}^2 a_{31}^{\vphantom{1}} a_{33}^2
		- 6 a_{11}^2 a_{12}^{\vphantom{1}} a_{13}^{\vphantom{1}} a_{21}^{\vphantom{1}} a_{23}^2 a_{31}^{\vphantom{1}} a_{32}^2 \\
		& + 6 a_{11}^2 a_{12}^{\vphantom{1}} a_{13}^{\vphantom{1}} a_{22}^2 a_{23}^{\vphantom{1}} a_{31}^2 a_{33}^{\vphantom{1}}
		- 6 a_{11}^2 a_{12}^{\vphantom{1}} a_{13}^{\vphantom{1}} a_{22}^{\vphantom{1}} a_{23}^2 a_{31}^2 a_{32}^{\vphantom{1}}
		- a_{11}^2 a_{13}^2 a_{21}^2 a_{22}^{\vphantom{1}} a_{32}^2 a_{33}^{\vphantom{1}} \\
		& + 6 a_{11}^2 a_{13}^2 a_{21}^{\vphantom{1}} a_{22}^2 a_{31}^{\vphantom{1}} a_{32}^{\vphantom{1}} a_{33}^{\vphantom{1}}
		- 6 a_{11}^2 a_{13}^2 a_{21}^{\vphantom{1}} a_{22}^{\vphantom{1}} a_{23}^{\vphantom{1}} a_{31}^{\vphantom{1}} a_{32}^2
		+ a_{11}^2 a_{13}^2 a_{22}^2 a_{23}^{\vphantom{1}} a_{31}^2 a_{32}^{\vphantom{1}} \\
		& - 2 a_{11}^{\vphantom{1}} a_{12}^2 a_{13}^{\vphantom{1}} a_{21}^2 a_{22}^{\vphantom{1}} a_{31}^{\vphantom{1}} a_{33}^2
		+ 2 a_{11}^{\vphantom{1}} a_{12}^2 a_{13}^{\vphantom{1}} a_{21}^{\vphantom{1}} a_{23}^2 a_{31}^2 a_{32}^{\vphantom{1}}
		+ 2 a_{11}^{\vphantom{1}} a_{12}^{\vphantom{1}} a_{13}^2 a_{21}^2 a_{23}^{\vphantom{1}} a_{31}^{\vphantom{1}} a_{32}^2 \\
		& - 2 a_{11}^{\vphantom{1}} a_{12}^{\vphantom{1}} a_{13}^2 a_{21}^{\vphantom{1}} a_{22}^2 a_{31}^2 a_{33}^{\vphantom{1}}
		- 3 a_{12}^2 a_{13}^2 a_{21}^2 a_{22}^{\vphantom{1}} a_{31}^2 a_{33}^{\vphantom{1}}
		+ 3 a_{12}^2 a_{13}^2 a_{21}^2 a_{23}^{\vphantom{1}} a_{31}^2 a_{32}^{\vphantom{1}}
	\big) \\
	b_3 = {} & 2 a_{11}^2 \big(
		{-}10 a_{11}^3 a_{22}^2 a_{23}^{\vphantom{1}} a_{32}^{\vphantom{1}} a_{33}^2
		+ 10 a_{11}^3 a_{22}^{\vphantom{1}} a_{23}^2 a_{32}^2 a_{33}^{\vphantom{1}}
		+ 4 a_{11}^2 a_{12}^{\vphantom{1}} a_{21}^{\vphantom{1}} a_{22}^{\vphantom{1}} a_{23}^{\vphantom{1}} a_{32}^{\vphantom{1}} a_{33}^2 \\
		& - 5 a_{11}^2 a_{12}^{\vphantom{1}} a_{21}^{\vphantom{1}} a_{23}^2 a_{32}^2 a_{33}^{\vphantom{1}}
		+ 5 a_{11}^2 a_{12}^{\vphantom{1}} a_{22}^2 a_{23}^{\vphantom{1}} a_{31}^{\vphantom{1}} a_{33}^2
		- 4 a_{11}^2 a_{12}^{\vphantom{1}} a_{22}^{\vphantom{1}} a_{23}^2 a_{31}^{\vphantom{1}} a_{32}^{\vphantom{1}} a_{33}^{\vphantom{1}} \\
		& + 5 a_{11}^2 a_{13}^{\vphantom{1}} a_{21}^{\vphantom{1}} a_{22}^2 a_{32}^{\vphantom{1}} a_{33}^2
		- 4 a_{11}^2 a_{13}^{\vphantom{1}} a_{21}^{\vphantom{1}} a_{22}^{\vphantom{1}} a_{23}^{\vphantom{1}} a_{32}^2 a_{33}^{\vphantom{1}}
		+ 4 a_{11}^2 a_{13}^{\vphantom{1}} a_{22}^2 a_{23}^{\vphantom{1}} a_{31}^{\vphantom{1}} a_{32}^{\vphantom{1}} a_{33}^{\vphantom{1}} \\
		& - 5 a_{11}^2 a_{13}^{\vphantom{1}} a_{22}^{\vphantom{1}} a_{23}^2 a_{31}^{\vphantom{1}} a_{32}^2
		- a_{11}^{\vphantom{1}} a_{12}^2 a_{21}^{\vphantom{1}} a_{22}^{\vphantom{1}} a_{23}^{\vphantom{1}} a_{31}^{\vphantom{1}} a_{33}^2
		+ a_{11}^{\vphantom{1}} a_{12}^2 a_{21}^{\vphantom{1}} a_{23}^2 a_{31}^{\vphantom{1}} a_{32}^{\vphantom{1}} a_{33}^{\vphantom{1}} \\
		& - a_{11}^{\vphantom{1}} a_{12}^{\vphantom{1}} a_{13}^{\vphantom{1}} a_{21}^2 a_{22}^{\vphantom{1}} a_{32}^{\vphantom{1}} a_{33}^2
		+ a_{11}^{\vphantom{1}} a_{12}^{\vphantom{1}} a_{13}^{\vphantom{1}} a_{21}^2 a_{23}^{\vphantom{1}} a_{32}^2 a_{33}^{\vphantom{1}}
		- 2 a_{11}^{\vphantom{1}} a_{12}^{\vphantom{1}} a_{13}^{\vphantom{1}} a_{21}^{\vphantom{1}} a_{22}^2 a_{31}^{\vphantom{1}} a_{33}^2 \\
		& + 2 a_{11}^{\vphantom{1}} a_{12}^{\vphantom{1}} a_{13}^{\vphantom{1}} a_{21}^{\vphantom{1}} a_{23}^2 a_{31}^{\vphantom{1}} a_{32}^2
		- a_{11}^{\vphantom{1}} a_{12}^{\vphantom{1}} a_{13}^{\vphantom{1}} a_{22}^2 a_{23}^{\vphantom{1}} a_{31}^2 a_{33}^{\vphantom{1}}
		+ a_{11}^{\vphantom{1}} a_{12}^{\vphantom{1}} a_{13}^{\vphantom{1}} a_{22}^{\vphantom{1}} a_{23}^2 a_{31}^2 a_{32}^{\vphantom{1}} \\
		& - a_{11}^{\vphantom{1}} a_{13}^2 a_{21}^{\vphantom{1}} a_{22}^2 a_{31}^{\vphantom{1}} a_{32}^{\vphantom{1}} a_{33}^{\vphantom{1}}
		+ a_{11}^{\vphantom{1}} a_{13}^2 a_{21}^{\vphantom{1}} a_{22}^{\vphantom{1}} a_{23}^{\vphantom{1}} a_{31}^{\vphantom{1}} a_{32}^2
		+ a_{12}^2 a_{13}^{\vphantom{1}} a_{21}^2 a_{23}^{\vphantom{1}} a_{31}^{\vphantom{1}} a_{32}^{\vphantom{1}} a_{33}^{\vphantom{1}} \\
		& - a_{12}^2 a_{13}^{\vphantom{1}} a_{21}^{\vphantom{1}} a_{22}^{\vphantom{1}} a_{23}^{\vphantom{1}} a_{31}^2 a_{33}^{\vphantom{1}}
		- a_{12}^{\vphantom{1}} a_{13}^2 a_{21}^2 a_{22}^{\vphantom{1}} a_{31}^{\vphantom{1}} a_{32}^{\vphantom{1}} a_{33}^{\vphantom{1}}
		+ a_{12}^{\vphantom{1}} a_{13}^2 a_{21}^{\vphantom{1}} a_{22}^{\vphantom{1}} a_{23}^{\vphantom{1}} a_{31}^2 a_{32}^{\vphantom{1}}
	\big) \\
	b_2 = {} & a_{11}^3 \big(
		15 a_{11}^2 a_{22}^2 a_{23}^{\vphantom{1}} a_{32}^{\vphantom{1}} a_{33}^2
		- 15 a_{11}^2 a_{22}^{\vphantom{1}} a_{23}^2 a_{32}^2 a_{33}^{\vphantom{1}}
		- 2 a_{11}^{\vphantom{1}} a_{12}^{\vphantom{1}} a_{21}^{\vphantom{1}} a_{22}^{\vphantom{1}} a_{23}^{\vphantom{1}} a_{32}^{\vphantom{1}} a_{33}^2 \\
		& + 5 a_{11}^{\vphantom{1}} a_{12}^{\vphantom{1}} a_{21}^{\vphantom{1}} a_{23}^2 a_{32}^2 a_{33}^{\vphantom{1}}
		- 5 a_{11}^{\vphantom{1}} a_{12}^{\vphantom{1}} a_{22}^2 a_{23}^{\vphantom{1}} a_{31}^{\vphantom{1}} a_{33}^2
		+ 2 a_{11}^{\vphantom{1}} a_{12}^{\vphantom{1}} a_{22}^{\vphantom{1}} a_{23}^2 a_{31}^{\vphantom{1}} a_{32}^{\vphantom{1}} a_{33}^{\vphantom{1}} \\
		& - 5 a_{11}^{\vphantom{1}} a_{13}^{\vphantom{1}} a_{21}^{\vphantom{1}} a_{22}^2 a_{32}^{\vphantom{1}} a_{33}^2
		+ 2 a_{11}^{\vphantom{1}} a_{13}^{\vphantom{1}} a_{21}^{\vphantom{1}} a_{22}^{\vphantom{1}} a_{23}^{\vphantom{1}} a_{32}^2 a_{33}^{\vphantom{1}}
		- 2 a_{11}^{\vphantom{1}} a_{13}^{\vphantom{1}} a_{22}^2 a_{23}^{\vphantom{1}} a_{31}^{\vphantom{1}} a_{32}^{\vphantom{1}} a_{33}^{\vphantom{1}} \\
		& + 5 a_{11}^{\vphantom{1}} a_{13}^{\vphantom{1}} a_{22}^{\vphantom{1}} a_{23}^2 a_{31}^{\vphantom{1}} a_{32}^2
		+ a_{12}^{\vphantom{1}} a_{13}^{\vphantom{1}} a_{21}^{\vphantom{1}} a_{22}^2 a_{31}^{\vphantom{1}} a_{33}^2
		- a_{12}^{\vphantom{1}} a_{13}^{\vphantom{1}} a_{21}^{\vphantom{1}} a_{23}^2 a_{31}^{\vphantom{1}} a_{32}^2
	\big) \\
	b_1 = {} & a_{11}^4 \big(
		{-}6 a_{11}^{\vphantom{1}} a_{22}^2 a_{23}^{\vphantom{1}} a_{32}^{\vphantom{1}} a_{33}^2
		+ 6 a_{11}^{\vphantom{1}} a_{22}^{\vphantom{1}} a_{23}^2 a_{32}^2 a_{33}^{\vphantom{1}}
		- a_{12}^{\vphantom{1}} a_{21}^{\vphantom{1}} a_{23}^2 a_{32}^2 a_{33}^{\vphantom{1}} \\
		& + a_{12}^{\vphantom{1}} a_{22}^2 a_{23}^{\vphantom{1}} a_{31}^{\vphantom{1}} a_{33}^2
		+ a_{13}^{\vphantom{1}} a_{21}^{\vphantom{1}} a_{22}^2 a_{32}^{\vphantom{1}} a_{33}^2
		- a_{13}^{\vphantom{1}} a_{22}^{\vphantom{1}} a_{23}^2 a_{31}^{\vphantom{1}} a_{32}^2
	\big) \\
	b_0 = {} &
		a_{11}^5
		a_{22}
		a_{23}
		a_{32}
		a_{33}
		\left(a_{22} a_{33} - a_{23} a_{32}\right).
\end{align*}
\normalsize
\caption{Coefficients in the polynomial equation satisfied by the top-left entry of the Sinkhorn limit of a positive $3 \times 3$ matrix.}
\label{3 by 3 coefficients}
\end{table}

In particular, the degree of each entry of $\Sink(A)$ is at most $6$.
For the matrix~$A$ in Example~\ref{3 by 3 example}, Theorem~\ref{3 by 3 polynomial - large} states that the top-left entry $x$ of $\Sink(A)$ satisfies
\[
	81 \left(374752 x^6 - 220388 x^5 - 844359 x^4 - 125796 x^3 + 210897 x^2 + 14346 x - 12312\right) = 0.
\]
This agrees with Equation~\eqref{3 by 3 example matrix polynomial}.
This polynomial is irreducible.
Along with Theorem~\ref{3 by 3 polynomial - large}, this confirms the conjecture of Chen and Varghese~\cite{Chen--Varghese} that the entries of $\Sink(A)$ for positive $3 \times 3$ matrices $A$ generically have degree $6$.

Let $f(x) = b_6 x^6 + b_5 x^5 + \dots + b_1 x + b_0$ be the polynomial in Theorem~\ref{3 by 3 polynomial - large}.
When we project $A$ to a symmetric matrix by setting $a_{21} = a_{12}$, $a_{31} = a_{13}$, and $a_{32} = a_{23}$,
then the projection of $f(x)$ factors as $-\left((a_{11} a_{23} - a_{12} a_{13}) x - a_{11} a_{23}\right)^2 g(x)$, where $g(x)$ is the degree-$4$ polynomial computed by Ekhad and Zeilberger~\cite{Ekhad--Zeilberger} for the top-left entry.

An obvious question is whether there is a better way to write the polynomial $f(x)$ in Theorem~\ref{3 by 3 polynomial - large}.
In fact there is, using determinants.
We first observe that $b_k$ contains the factor $a_{11}^{5 - k}$ for each $k \in \{0, 1, \dots, 5\}$.
This suggests that the scaled polynomial $a_{11} f(x)$ is more natural than $f(x)$, since the coefficient of $x^k$ in $a_{11} f(x)$ contains the factor $a_{11}^{6 - k}$ not just for $k \in \{0, 1, \dots, 5\}$ but for all $k \in \{0, 1, \dots, 6\}$.

The leading coefficient $a_{11} b_6$ of $a_{11} f(x)$ is the product of $6$ minors of $A$.
More specifically, it is the product of all the minors of $A$ that involve $a_{11}$.
Furthermore, the constant coefficient $a_{11} b_0$ of $a_{11} f(x)$ is the product of $a_{11}^6$ and the $6$ minors of $A$ that do not involve the first row or first column (including the determinant $1$ of the $0 \times 0$ matrix).

To rewrite the other coefficients, we would like to interpolate between these products for $a_{11} b_6$ and $a_{11} b_0$.
One possibility is that $a_{11} b_k$ is a linear combination of products of minors, where each product consists of
\begin{itemize}
\item $k$ minors that involve the first row and first column,
\item $6 - k$ minors that do not involve the first row or first column, and
\item $a_{11}^{6 - k}$ (equivalently, one factor $a_{11}$ for each of the latter).
\end{itemize}
We have seen that this is the case for $a_{11} b_6$ and $a_{11} b_0$.
Remarkably, it turns out that $a_{11} b_5, a_{11} b_4, \dots, a_{11} b_1$ can be written this way as well.
We use the notation $\Delta$, $\Gamma$, and $M$ established in Section~\ref{Introduction}.
To make a subset $S = \{(R_1, C_1), (R_2, C_2), \dots, (R_k, C_k)\}$ of $D(m, n)$ easier to read, we format it as
\[
	S =
	\subsetmatrix{
		R_1 & R_2 & \cdots & R_k\vphantom{\{} \\
		C_1 & C_2 & \cdots & C_k\vphantom{\{}
	}.
\]

\begin{example}\label{notation example}
For $m = 3$ and $n = 3$, we have
$
	D(3, 3) =
		\subsetmatrix{
			\{\} & \{2\} & \{2\} & \{3\} & \{3\} & \{2, 3\} \\
			\{\} & \{2\} & \{3\} & \{2\} & \{3\} & \{2, 3\}
		}
$.
Moreover, from the expressions for $b_6$ and $b_0$ in Table~\ref{3 by 3 coefficients}, we have
\begin{align*}
	a_{11} b_6 &=
		\Delta\!\left(\!\begin{smallmatrix}
			\{\} \\
			\{\}
		\end{smallmatrix}\!\right)
		\Delta\!\left(\!\begin{smallmatrix}
			\{2\} \\
			\{2\}
		\end{smallmatrix}\!\right)
		\Delta\!\left(\!\begin{smallmatrix}
			\{2\} \\
			\{3\}
		\end{smallmatrix}\!\right)
		\Delta\!\left(\!\begin{smallmatrix}
			\{3\} \\
			\{2\}
		\end{smallmatrix}\!\right)
		\Delta\!\left(\!\begin{smallmatrix}
			\{3\} \\
			\{3\}
		\end{smallmatrix}\!\right)
		\Delta\!\left(\!\begin{smallmatrix}
			\{2, 3\} \\
			\{2, 3\}
		\end{smallmatrix}\!\right)
		=
		M\!\left(\!\subsetmatrix{
			\{\} & \{2\} & \{2\} & \{3\} & \{3\} & \{2, 3\} \\
			\{\} & \{2\} & \{3\} & \{2\} & \{3\} & \{2, 3\}
		}\!\right) \\
	a_{11} b_0 &=
		\Gamma\!\left(\!\begin{smallmatrix}
			\{\} \\
			\{\}
		\end{smallmatrix}\!\right)
		\Gamma\!\left(\!\begin{smallmatrix}
			\{2\} \\
			\{2\}
		\end{smallmatrix}\!\right)
		\Gamma\!\left(\!\begin{smallmatrix}
			\{2\} \\
			\{3\}
		\end{smallmatrix}\!\right)
		\Gamma\!\left(\!\begin{smallmatrix}
			\{3\} \\
			\{2\}
		\end{smallmatrix}\!\right)
		\Gamma\!\left(\!\begin{smallmatrix}
			\{3\} \\
			\{3\}
		\end{smallmatrix}\!\right)
		\Gamma\!\left(\!\begin{smallmatrix}
			\{2, 3\} \\
			\{2, 3\}
		\end{smallmatrix}\!\right)
		=
		M\!\left(\!\subsetmatrix{
			\vphantom{\{} \\
			\vphantom{\{}
		}\!\right).
\end{align*}
Each of these monomials involves $\Delta$ or $\Gamma$ but not both, whereas the monomial
\[
	M\!\left(\!\subsetmatrix{
		\{\} & \{3\} \\
		\{\} & \{2\}
	}\!\right)
	=
		\Delta\!\left(\!\begin{smallmatrix}
			\{\} \\
			\{\}
		\end{smallmatrix}\!\right)
		\Gamma\!\left(\!\begin{smallmatrix}
			\{2\} \\
			\{2\}
		\end{smallmatrix}\!\right)
		\Gamma\!\left(\!\begin{smallmatrix}
			\{2\} \\
			\{3\}
		\end{smallmatrix}\!\right)
		\Delta\!\left(\!\begin{smallmatrix}
			\{3\} \\
			\{2\}
		\end{smallmatrix}\!\right)
		\Gamma\!\left(\!\begin{smallmatrix}
			\{3\} \\
			\{3\}
		\end{smallmatrix}\!\right)
		\Gamma\!\left(\!\begin{smallmatrix}
			\{2, 3\} \\
			\{2, 3\}
		\end{smallmatrix}\!\right)
\]
involves both $\Delta$ and $\Gamma$, for example.
\end{example}

We continue to let $A$ be the general $3 \times 3$ matrix in Equation~\eqref{3 by 3 matrix}, and let $x$ be the top-left entry of $\Sink(A)$.
To rewrite the coefficients $a_{11} b_k$, we look for coefficients $c_S \in \Q$, indexed by subsets $S \subseteq D(3, 3)$, such that
\[
	a_{11} b_k = \sum_{\substack{S \subseteq D(3, 3) \\ \size{S} = k}} c_S M(S)
\]
for each $k \in \{0, 1, \dots, 6\}$.
This will give the polynomial equation
\[
	\sum_{k = 0}^{6} \left(\sum_{\substack{S \subseteq D(3, 3) \\ \size{S} = k}} c_S M(S)\right) x^k = 0.
\]
The expressions for $a_{11} b_6$ and $a_{11} b_0$ in Example~\ref{notation example} allow us to choose $c_{D(3, 3)} = 1$ and $c_{\{\}} = 1$.
For each $k \in \{1, 2, 4, 5\}$, one checks that $a_{11} b_k$ can be written uniquely as a linear combination of the monomials $M(S)$ where $\size{S} = k$.
This determines $c_S$ for such subsets $S$.
For example,
\begin{equation}\label{a11 b1 linear combination}
	a_{11} b_1 =
		-3 M\!\left(\!\subsetmatrix{
			\{\} \\
			\{\}
		}\!\right)
		- M\!\left(\!\subsetmatrix{
			\{2\} \\
			\{2\}
		}\!\right)
		- M\!\left(\!\subsetmatrix{
			\{2\} \\
			\{3\}
		}\!\right)
		- M\!\left(\!\subsetmatrix{
			\{3\} \\
			\{2\}
		}\!\right)
		- M\!\left(\!\subsetmatrix{
			\{3\} \\
			\{3\}
		}\!\right)
		+ M\!\left(\!\subsetmatrix{
			\{2, 3\} \\
			\{2, 3\}
		}\!\right).
\end{equation}

However, the coefficient $a_{11} b_3$ has multiple representations.
The family of such representations is $1$-dimensional, due to the relation
\begin{multline*}
	M\!\left(\!\subsetmatrix{
		\{\} & \{2\} & \{3\} \\
		\{\} & \{2\} & \{3\}
	}\!\right)
	+ M\!\left(\!\subsetmatrix{
		\{\} & \{2\} & \{2, 3\} \\
		\{\} & \{3\} & \{2, 3\}
	}\!\right)
	+ M\!\left(\!\subsetmatrix{
		\{\} & \{3\} & \{2, 3\} \\
		\{\} & \{2\} & \{2, 3\}
	}\!\right) \\
	+ M\!\left(\!\subsetmatrix{
		\{2\} & \{2\} & \{3\} \\
		\{2\} & \{3\} & \{2\}
	}\!\right)
	+ M\!\left(\!\subsetmatrix{
		\{2\} & \{3\} & \{2, 3\} \\
		\{2\} & \{3\} & \{2, 3\}
	}\!\right)
	+ M\!\left(\!\subsetmatrix{
		\{2\} & \{3\} & \{3\} \\
		\{3\} & \{2\} & \{3\}
	}\!\right) \\
	=
	M\!\left(\!\subsetmatrix{
		\{\} & \{2\} & \{2, 3\} \\
		\{\} & \{2\} & \{2, 3\}
	}\!\right)
	+ M\!\left(\!\subsetmatrix{
		\{\} & \{2\} & \{3\} \\
		\{\} & \{3\} & \{2\}
	}\!\right)
	+ M\!\left(\!\subsetmatrix{
		\{\} & \{3\} & \{2, 3\} \\
		\{\} & \{3\} & \{2, 3\}
	}\!\right) \\
	+ M\!\left(\!\subsetmatrix{
		\{2\} & \{2\} & \{3\} \\
		\{2\} & \{3\} & \{3\}
	}\!\right)
	+ M\!\left(\!\subsetmatrix{
		\{2\} & \{3\} & \{3\} \\
		\{2\} & \{2\} & \{3\}
	}\!\right)
	+ M\!\left(\!\subsetmatrix{
		\{2\} & \{3\} & \{2, 3\} \\
		\{3\} & \{2\} & \{2, 3\}
	}\!\right).
\end{multline*}
Therefore the expression for $b_3$ in Table~\ref{3 by 3 coefficients} does not uniquely determine the coefficients $c_S$ where $\size{S} = 3$.
However, there is additional information we can use to obtain uniqueness.
The polynomial $a_{11} f(x)$ possesses symmetries arising from the following invariance properties.

Since the iterative scaling process isn't sensitive to row order, the top-left entry of $\Sink(A)$ is invariant under row permutations of $A$ that fix the first row.
Similarly for column permutations.
Additionally, the top-left entry of $\Sink(A)$ is invariant under transposition of $A$; this follows from Sinkhorn's result that $\Sink(A)$ is the unique doubly stochastic matrix $S$ such that $S = D_1 A D_2$ for some diagonal matrices $D_1$ and $D_2$.
This suggests the following equivalence relation.

\begin{notation*}
Let $S$ and $T$ be subsets of $D(m, n)$.
We write $T \equiv S$ if the set of minors specified by $T$ is transformed into the set of minors specified by $S$ by some composition of
\begin{itemize}
\item row permutations that fix the first row,
\item column permutations that fix the first column, and
\item transposition if $m = n$.
\end{itemize}
For each $S \subseteq D(m, n)$, define the \emph{class sum}
\[
	\Sigma(S) = \sum_{\substack{T \subseteq D(m, n) \\ T \equiv S}} M(T)
\]
to be the sum of the monomials corresponding to the elements in the equivalence class of $S$.
\end{notation*}

\begin{example}\label{equivalence class example}
Let $m = 3$ and $n = 3$, and consider the subset
$S =
	\subsetmatrix{
		\{2\} \\
		\{2\}
	}
$ of size $1$.
The equivalence class of $S$ is
\[
	\left\{
	\subsetmatrix{
		\{2\} \\
		\{2\}
	}, \,
	\subsetmatrix{
		\{2\} \\
		\{3\}
	}, \,
	\subsetmatrix{
		\{3\} \\
		\{2\}
	}, \,
	\subsetmatrix{
		\{3\} \\
		\{3\}
	}
	\right\}
\]
since these $4$ specifications of $1 \times 1$ minors can be obtained from each other by row and column permutations.
This equivalence class is reflected in the linear combination~\eqref{a11 b1 linear combination} for $a_{11} b_1$.
Namely, the four monomials
\[
	M\!\left(\!\subsetmatrix{
		\{2\} \\
		\{2\}
	}\!\right)\!, \,
	M\!\left(\!\subsetmatrix{
		\{2\} \\
		\{3\}
	}\!\right)\!, \,
	M\!\left(\!\subsetmatrix{
		\{3\} \\
		\{2\}
	}\!\right)\!, \,
	M\!\left(\!\subsetmatrix{
		\{3\} \\
		\{3\}
	}\!\right)
\]
all have the same coefficient $-1$.
In particular, their contribution to $a_{11} b_1$ is
\[
	-\Sigma\!\left(\!\subsetmatrix{
		\{2\} \\
		\{2\}
	}\!\right)
	=
	-M\!\left(\!\subsetmatrix{
		\{2\} \\
		\{2\}
	}\!\right)
	-
	M\!\left(\!\subsetmatrix{
		\{2\} \\
		\{3\}
	}\!\right)
	-
	M\!\left(\!\subsetmatrix{
		\{3\} \\
		\{2\}
	}\!\right)
	-
	M\!\left(\!\subsetmatrix{
		\{3\} \\
		\{3\}
	}\!\right).
\]
\end{example}

Motivated by Example~\ref{equivalence class example}, we add the constraint that $c_T = c_S$ when $T \equiv S$, so that the coefficients $c_S$ reflect the symmetries of $a_{11} f(x)$.
With this constraint, the coefficient $a_{11} b_3$ has a unique representation as a linear combination of monomials $M(S)$ where $\size{S} = 3$.
Writing each coefficient in Table~\ref{3 by 3 coefficients} using class sums gives the following improvement of Theorem~\ref{3 by 3 polynomial - large}.
Here $d_k = a_{11} b_k$.

\begin{theorem}\label{3 by 3 polynomial}
Let $A$ be a positive $3 \times 3$ matrix.
The top-left entry $x$ of $\Sink(A)$ satisfies $d_6 x^6 + d_5 x^5 + d_4 x^4 + d_3 x^3 + d_2 x^2 + d_1 x + d_0 = 0$, where
\begin{align*}
	d_6 &=
		\Sigma\!\left(\!\subsetmatrix{
			\{\} & \{2\} & \{2\} & \{3\} & \{3\} & \{2, 3\} \\
			\{\} & \{2\} & \{3\} & \{2\} & \{3\} & \{2, 3\}
		}\!\right) \\
	d_5 &=
		-3 \Sigma\!\left(\!\subsetmatrix{
			\{\} & \{2\} & \{2\} & \{3\} & \{3\} \\
			\{\} & \{2\} & \{3\} & \{2\} & \{3\}
		}\!\right)
		- \Sigma\!\left(\!\subsetmatrix{
			\{\} & \{2\} & \{2\} & \{3\} & \{2, 3\} \\
			\{\} & \{2\} & \{3\} & \{2\} & \{2, 3\}
		}\!\right)
		+ \Sigma\!\left(\!\subsetmatrix{
			\{2\} & \{2\} & \{3\} & \{3\} & \{2, 3\} \\
			\{2\} & \{3\} & \{2\} & \{3\} & \{2, 3\}
		}\!\right) \\
	d_4 &=
		4 \Sigma\!\left(\!\subsetmatrix{
			\{\} & \{2\} & \{2\} & \{3\} \\
			\{\} & \{2\} & \{3\} & \{2\}
		}\!\right)
		+ \Sigma\!\left(\!\subsetmatrix{
			\{\} & \{2\} & \{3\} & \{2, 3\} \\
			\{\} & \{2\} & \{3\} & \{2, 3\}
		}\!\right)
		- 3 \Sigma\!\left(\!\subsetmatrix{
			\{2\} & \{2\} & \{3\} & \{3\} \\
			\{2\} & \{3\} & \{2\} & \{3\}
		}\!\right) \\
	d_3 &=
		-4 \Sigma\!\left(\!\subsetmatrix{
			\{\} & \{2\} & \{2\} \\
			\{\} & \{2\} & \{3\}
		}\!\right)
		- 5 \Sigma\!\left(\!\subsetmatrix{
			\{\} & \{2\} & \{3\} \\
			\{\} & \{2\} & \{3\}
		}\!\right)
		+ \Sigma\!\left(\!\subsetmatrix{
			\{\} & \{2\} & \{2, 3\} \\
			\{\} & \{2\} & \{2, 3\}
		}\!\right)
		+ \Sigma\!\left(\!\subsetmatrix{
			\{2\} & \{2\} & \{3\} \\
			\{2\} & \{3\} & \{2\}
		}\!\right)
		- \Sigma\!\left(\!\subsetmatrix{
			\{2\} & \{3\} & \{2, 3\} \\
			\{2\} & \{3\} & \{2, 3\}
		}\!\right) \\
	d_2 &=
		4 \Sigma\!\left(\!\subsetmatrix{
			\{\} & \{2\} \\
			\{\} & \{2\}
		}\!\right)
		- 3 \Sigma\!\left(\!\subsetmatrix{
			\{\} & \{2, 3\} \\
			\{\} & \{2, 3\}
		}\!\right)
		+ \Sigma\!\left(\!\subsetmatrix{
			\{2\} & \{3\} \\
			\{2\} & \{3\}
		}\!\right) \\
	d_1 &=
		-3 \Sigma\!\left(\!\subsetmatrix{
			\{\} \\
			\{\}
		}\!\right)
		- \Sigma\!\left(\!\subsetmatrix{
			\{2\} \\
			\{2\}
		}\!\right)
		+ \Sigma\!\left(\!\subsetmatrix{
			\{2, 3\} \\
			\{2, 3\}
		}\!\right) \\
	d_0 &=
		\Sigma\!\left(\!\subsetmatrix{
			\vphantom{\{} \\
			\vphantom{\{}
		}\!\right).
\end{align*}
\end{theorem}

Several class sums do not appear in Theorem~\ref{3 by 3 polynomial}, namely
\[
	\Sigma\!\left(\!\subsetmatrix{
		\{\} & \{2\} & \{2\} & \{2, 3\} \\
		\{\} & \{2\} & \{3\} & \{2, 3\}
	}\!\right)\!, \,
	\Sigma\!\left(\!\subsetmatrix{
		\{2\} & \{2\} & \{3\} & \{2, 3\} \\
		\{2\} & \{3\} & \{2\} & \{2, 3\}
	}\!\right)\!, \,
	\Sigma\!\left(\!\subsetmatrix{
		\{2\} & \{2\} & \{2, 3\} \\
		\{2\} & \{3\} & \{2, 3\}
	}\!\right)\!, \,
	\Sigma\!\left(\!\subsetmatrix{
		\{2\} & \{2\} \\
		\{2\} & \{3\}
	}\!\right)\!, \,
	\Sigma\!\left(\!\subsetmatrix{
		\{2\} & \{2, 3\} \\
		\{2\} & \{2, 3\}
	}\!\right)\!.
\]
These class sums get assigned the coefficient $0$ when the coefficients $d_4, d_3, d_2$ are written as linear combinations of $\Sigma(S)$.
In total, there are $24$ equivalence classes of subsets $S \subseteq D(3, 3)$, so we have compressed the information in Table~\ref{3 by 3 coefficients} down to a function from the set of these $24$ equivalence classes to the set $\{-5, -4, -3, -1, 0, 1, 4\}$.

For the particular matrix~$A$ in Example~\ref{3 by 3 example}, Theorem~\ref{3 by 3 polynomial} gives $91064736 x^6 - 53554284 x^5 - 205179237 x^4 - 30568428 x^3 + 51247971 x^2 + 3486078 x - 2991816 = 0$ for the top-left entry.
This equation can be computed quickly with \textsc{SinkhornPolynomials}~\cite{SinkhornPolynomials}.
Dividing this equation by $243$ produces Equation~\eqref{3 by 3 example matrix polynomial}.

Some special cases can be obtained from Theorem~\ref{3 by 3 polynomial}.
For example, the following corollary shows that the degree drops if a minor is $0$.
(If multiple minors are $0$, the degree can drop further.)

\begin{corollary}
Let $A$ be a positive $3 \times 3$ matrix, and let $a_{11}$ be the top-left entry of $A$.
If one of the $2 \times 2$ minors involving $a_{11}$ is $0$ and all minors not involving $a_{11}$ are not $0$, then the top-left entry of $\Sink(A)$ has degree at most $5$.
\end{corollary}

\begin{proof}
The coefficient of $x^6$ in Theorem~\ref{3 by 3 polynomial} is
\[
	d_6
	=
	\Sigma\!\left(\!\subsetmatrix{
		\{\} & \{2\} & \{2\} & \{3\} & \{3\} & \{2, 3\} \\
		\{\} & \{2\} & \{3\} & \{2\} & \{3\} & \{2, 3\}
	}\!\right)
	=
	M\!\left(\!\subsetmatrix{
		\{\} & \{2\} & \{2\} & \{3\} & \{3\} & \{2, 3\} \\
		\{\} & \{2\} & \{3\} & \{2\} & \{3\} & \{2, 3\}
	}\!\right)
	=
	\prod_{(R, C) \in D(3, 3)}
	\Delta\!\left(\!\begin{smallmatrix}
		R\vphantom{\{} \\
		C\vphantom{\{}
	\end{smallmatrix}\!\right).
\]
Since one of the minors involving $a_{11}$ is $0$, this product is $0$, so $d_6 = 0$.
On the other hand, since all minors not involving $a_{11}$ are not $0$, we have
\[
	d_0
	=
	\Sigma\!\left(\!\subsetmatrix{
		\vphantom{\{} \\
		\vphantom{\{}
	}\!\right)
	=
	M\!\left(\!\subsetmatrix{
		\vphantom{\{} \\
		\vphantom{\{}
	}\!\right)
	=
	\prod_{(R, C) \in D(3, 3)}
	\Gamma\!\left(\!\begin{smallmatrix}
		R\vphantom{\{} \\
		C\vphantom{\{}
	\end{smallmatrix}\!\right)
	\neq 0.
\]
Therefore the polynomial in Theorem~\ref{3 by 3 polynomial} is not the $0$ polynomial, and its degree in $x$ is at most $5$.
\end{proof}

If a matrix is sufficiently degenerate, then Theorem~\ref{3 by 3 polynomial} is vacuously true and does not immediately give any information about $\Sink(A)$.
For example, let
\begin{equation}\label{3 by 3 example matrix with row multiples}
	A =
	\begin{bmatrix}
		4 & 5 & 6 \\
		1 & 2 & 3 \\
		2 & 4 & 6
	\end{bmatrix}.
\end{equation}
Here rows $2$ and $3$ are scalar multiples of each other, so
$
	\Delta\!\left(\!\begin{smallmatrix}
		\{2, 3\} \\
		\{2, 3\}
	\end{smallmatrix}\!\right)
	= \det A
	= 0
$
and
$
	\Gamma\!\left(\!\begin{smallmatrix}
		\{2, 3\} \\
		\{2, 3\}
	\end{smallmatrix}\!\right)
	= a_{11} \det A_{\{2, 3\}, \{2, 3\}}
	= 0
$.
Since each monomial $M(S)$ contains one of these two determinants as a factor, we have $M(S) = 0$ for all $S \subseteq D(3, 3)$.
Therefore $d_k = 0$ for each $k$.
However, we can still use Theorem~\ref{3 by 3 polynomial} to determine the entries of $\Sink(A)$, as the following result shows.

\begin{corollary}\label{3 by 3 matrix with row multiples}
Let $A$ be a positive $3 \times 3$ matrix, and let $a_{ij}$ be the $(i, j)$ entry of $A$.
If rows $2$ and $3$ are scalar multiples of each other, then the top-left entry $x$ of $\Sink(A)$ satisfies $e_3 x^3 + e_2 x^2 + e_1 x + e_0 = 0$, where
\begin{align*}
	e_3 &= a_{11} (a_{11} a_{22} - a_{12} a_{21}) (a_{11} a_{23} - a_{13} a_{21}) \\
	e_2 &= a_{11} (a_{11} a_{12} a_{21} a_{23} + a_{11} a_{13} a_{21} a_{22} - 3 a_{11}^2 a_{22} a_{23} + a_{12} a_{13} a_{21}^2) \\
	e_1 &= 3 a_{11}^3 a_{22} a_{23} \\
	e_0 &= -a_{11}^3 a_{22} a_{23}.
\end{align*}
In particular, $x$ is independent of row $3$.
\end{corollary}

The analogous result holds if columns $2$ and $3$ are scalar multiples of each other.
For the matrix in Equation~\eqref{3 by 3 example matrix with row multiples}, Corollary~\ref{3 by 3 matrix with row multiples} gives $24 (3 x^3 - 25 x^2 + 48 x - 16) = 0$.

\begin{proof}[Proof of Corollary~\ref{3 by 3 matrix with row multiples}]
The idea is to carefully factor out the minors that are $0$.
First we describe how to obtain the polynomial $e_3 x^3 + e_2 x^2 + e_1 x + e_0$, and then we justify it.

Begin with a general $3 \times 3$ matrix~$A$ with symbolic entries;
in particular, there are no algebraic relations between the entries.
Let $r, s$ be symbols; eventually we will set $r$ to be the scalar factor $\frac{a_{31}}{a_{21}}$.
Apply Theorem~\ref{3 by 3 polynomial} to $A$ to obtain a polynomial equation satisfied by the top-left entry of $\Sink(A)$.
Then replace each instance of $\det A$ in this polynomial with $s \det A_{\{2, 3\}, \{2, 3\}}$.
By the definition of $M(S)$, each monomial now contains $\det A_{\{2, 3\}, \{2, 3\}}$ as factor; divide by this factor.
Finally, replace $a_{3,j}$ with $r a_{2,j}$ for each $j \in \{1, 2, 3\}$.
The resulting polynomial factors as a product two cubic polynomials in $x$.
One of these cubic factors is independent of $r$ and $s$;
this is the polynomial $e_3 x^3 + e_2 x^2 + e_1 x + e_0$.

To justify the construction, let $A$ now be the matrix in the statement of the corollary, let $r = \frac{a_{31}}{a_{21}}$ be the scalar factor, and fix a real number $s \neq 0$.
We approximate $A$ by matrices with no $0$ minors.
Namely, let $B(t)$ be a continuous \textcolor{Gray}{(}$3 \times 3$ matrix\textcolor{Gray}{)}-valued function such that $\lim_{t \to 0^+} B(t) = A$ and, for all $t > 0$, no minor of $B(t)$ is $0$.
Further, we assume for all $t > 0$ that $\frac{\det B(t)}{\det B(t)_{\{2, 3\}, \{2, 3\}}} = s$.
Apply Theorem~\ref{3 by 3 polynomial} to $B(t)$ where $t > 0$.
Since $\det B(t) = s \det B(t)_{\{2, 3\}, \{2, 3\}}$ by assumption, the polynomial given by Theorem~\ref{3 by 3 polynomial} is of the form $(\det B(t)_{\{2, 3\}, \{2, 3\}}) g(t, x)$.
The entries of $\Sink(B(t))$ are continuous functions of the entries of $B(t)$, and the roots of a polynomial are continuous functions of its coefficients, so $g(x) \colonequal \lim_{t \to 0^+} g(t, x)$ is a polynomial for the top-left entry of $\Sink(A)$.
Moreover, since only one of the two cubic factors of $g(x)$ is independent of $s$, the top-left entry of $\Sink(A)$ is a root of that cubic factor.
\end{proof}

\section{Square matrices and solving for coefficients}\label{Square matrices}

In this section, we use Theorem~\ref{3 by 3 polynomial} to infer the form of an equation satisfied by the top-left entry of the Sinkhorn limit of a positive $n \times n$ matrix.
We then interpolate the coefficients in the equation for $4 \times 4$ matrices to obtain Conjecture~\ref{4 by 4 polynomial - large}.

For a positive $2 \times 2$ matrix $A$, the top-left entry $x$ of $\Sink(A)$ satisfies Equation~\eqref{2 by 2 polynomial}, namely $\left(a_{11} a_{22} - a_{12} a_{21}\right) x^2 - 2 a_{11} a_{22} \, x + a_{11} a_{22} = 0$.
Multiplying by $a_{11}$, this can be written using $\Delta$ and $\Gamma$ as
\begin{equation}\label{2 by 2 polynomial - general form}
	\Delta\!\left(\!\begin{smallmatrix}
		\{\} \\
		\{\}
	\end{smallmatrix}\!\right)
	\Delta\!\left(\!\begin{smallmatrix}
		\{2\} \\
		\{2\}
	\end{smallmatrix}\!\right)
	x^2
	- 2
	\Delta\!\left(\!\begin{smallmatrix}
		\{\} \\
		\{\}
	\end{smallmatrix}\!\right)
	\Gamma\!\left(\!\begin{smallmatrix}
		\{2\} \\
		\{2\}
	\end{smallmatrix}\!\right)
	x
	+
	\Gamma\!\left(\!\begin{smallmatrix}
		\{\} \\
		\{\}
	\end{smallmatrix}\!\right)
	\Gamma\!\left(\!\begin{smallmatrix}
		\{2\} \\
		\{2\}
	\end{smallmatrix}\!\right)
	= 0
\end{equation}
or, equivalently,
\begin{equation}\label{2 by 2 polynomial - compact general form}
	M\!\left(\!\subsetmatrix{
		\{\} & \{2\} \\
		\{\} & \{2\}
	}\!\right)
	x^2
	- 2
	M\!\left(\!\subsetmatrix{
		\{\} \\
		\{\}
	}\!\right)
	x
	+
	M\!\left(\!\subsetmatrix{
		\vphantom{\{} \\
		\vphantom{\{}
	}\!\right)
	= 0.
\end{equation}
Along with Theorem~\ref{3 by 3 polynomial}, this suggests that, for an $n \times n$ matrix, the coefficient of $x^k$ is a $\Z$-linear combination of the monomials $M(S)$ where $\size{S} = k$.
In particular, we expect the degree of the equation to be $\size{D(n, n)} = \sum_{k = 0}^{n - 1} \binom{n - 1}{k}^2 = \binom{2 n - 2}{n - 1}$, so that generically the entries of $\Sink(A)$ for a $4 \times 4$ matrix have degree $\binom{6}{3} = 20$ and for a $5 \times 5$ matrix have degree $\binom{8}{4} = 70$.
We generalize the notation for the coefficients $c_S$ from the previous section to $c_S(n)$ for an $n \times n$ matrix, where $c_S(3) \colonequal c_S$.

\begin{conjecture}\label{polynomial form}
Let $n \geq 1$.
There exist integers $c_S(n)$, indexed by subsets $S \subseteq D(n, n)$, such that, for every positive $n \times n$ matrix~$A$, the top-left entry $x$ of $\Sink(A)$ satisfies
\begin{equation}\label{sum over subsets}
	\sum_{k = 0}^{\binom{2 n - 2}{n - 1}} \left(\sum_{\substack{S \subseteq D(n, n) \\ \size{S} = k}} c_S(n) M(S)\right) x^k = 0
\end{equation}
or, equivalently,
\[
	\sum_{S \subseteq D(n, n)} c_S(n) M(S) x^{\size{S}} = 0.
\]
\end{conjecture}

This conjecture implies that, since both
$
	\Delta\!\left(\!\begin{smallmatrix}
		R\vphantom{\{} \\
		C\vphantom{\{}
	\end{smallmatrix}\!\right)
$ and $
	\Gamma\!\left(\!\begin{smallmatrix}
		R\vphantom{\{} \\
		C\vphantom{\{}
	\end{smallmatrix}\!\right)
$
are homogeneous degree-$(\size{R} + 1)$ polynomials in the entries of $A$, each monomial $M(S)$ (and therefore also the coefficient of each $x^k$ in Equation~\eqref{sum over subsets}) is a homogeneous polynomial with degree
\[
	\sum_{(R, C) \in D(n, n)} (\size{R} + 1)
	= \sum_{k = 0}^{n - 1} \binom{n - 1}{k}^2 (k + 1)
	= \frac{n + 1}{2} \binom{2 n - 2}{n - 1}.
\]
For $n = 1, 2, 3, \dots$, this degree is $1, 3, 12, 50, 210, 882, 3696, 15444, \dots$~\seq{A092443}.

The coefficients $c_S(n)$ in Conjecture~\ref{polynomial form} are not uniquely determined, since we can scale the polynomial.
To remove this source of non-uniqueness, we define $c_{\{\}}(n) = 1$ for all $n \geq 1$.
Based on the polynomials for $n = 2$ and $n = 3$, we conjecture that $c_{D(n, n)}(n) = 1$ for all $n \geq 2$.

We mention two surprising properties of the polynomial in Theorem~\ref{3 by 3 polynomial}.
These properties also hold for the quadratic polynomial in Equation~\eqref{2 by 2 polynomial - general form}, so we expect them to hold more generally for the polynomial in Conjecture~\ref{polynomial form}.
The first is a symmetry in the coefficients.
Each coefficient $d_k$ in Theorem~\ref{3 by 3 polynomial} can be obtained from $d_{6 - k}$ by replacing
\[
	\Delta\!\left(\!\begin{smallmatrix}
		R\vphantom{\{} \\
		C\vphantom{\{}
	\end{smallmatrix}\!\right)
	\mapsto
	\Gamma\!\left(\!\begin{smallmatrix}
		\{2, 3\} \setminus R \\
		\{2, 3\} \setminus C
	\end{smallmatrix}\!\right)
	\quad \text{and} \quad
	\Gamma\!\left(\!\begin{smallmatrix}
		R\vphantom{\{} \\
		C\vphantom{\{}
	\end{smallmatrix}\!\right)
	\mapsto
	\Delta\!\left(\!\begin{smallmatrix}
		\{2, 3\} \setminus R \\
		\{2, 3\} \setminus C
	\end{smallmatrix}\!\right).
\]

\begin{conjecture}\label{symmetry}
Let $n \geq 2$.
For each $(R, C) \in D(n, n)$, define
$
	\kappa\!\left(\!\begin{smallmatrix}
		R\vphantom{\{} \\
		C\vphantom{\{}
	\end{smallmatrix}\!\right)
	=
	\begin{smallmatrix}
		\{2, 3, \dots, n\} \setminus R \\
		\{2, 3, \dots, n\} \setminus C
	\end{smallmatrix}
$,
and extend $\kappa$ to subsets of $D(n, n)$ by
\[
	\kappa\!\left(\!\begin{smallmatrix}
		R_1 & R_2 & \cdots & R_k\vphantom{\{} \\
		C_1 & C_2 & \cdots & C_k\vphantom{\{}
	\end{smallmatrix}\!\right)
	= \left\{
		\kappa\!\left(\!\begin{smallmatrix}
			R_1\vphantom{\{} \\
			C_1\vphantom{\{}
		\end{smallmatrix}\!\right),
		\kappa\!\left(\!\begin{smallmatrix}
			R_2\vphantom{\{} \\
			C_2\vphantom{\{}
		\end{smallmatrix}\!\right),
		\dots,
		\kappa\!\left(\!\begin{smallmatrix}
			R_k\vphantom{\{} \\
			C_k\vphantom{\{}
		\end{smallmatrix}\!\right)
	\right\}.
\]
For each $S \subseteq D(n, n)$, we have $c_S(n) = c_{\kappa(D(n, n) \setminus S)}(n)$.
\end{conjecture}

The second is that, for each $k$, the sum of the coefficients $c_S(n)$ for $\size{S} = k$ is a signed binomial coefficient.
For example, the coefficients in the linear combination~\eqref{a11 b1 linear combination} for $d_1$ sum to $-3 - 1 - 1 - 1 - 1 + 1 = -6 = -\binom{6}{1}$.

\begin{conjecture}\label{binomial sum}
For each $n \geq 1$, we have
\[
	\sum_{S \subseteq D(n, n)} c_S(n) x^{\size{S}} = (x - 1)^{\binom{2 n - 2}{n - 1}}.
\]
\end{conjecture}

We do not have explanations for either of these properties.

It remains to determine the coefficients $c_S(n)$.
For $n = 3$, we determined the coefficients $c_S(3)$ from the output of a Gr\"obner basis computation, but for $n \geq 4$ this computation seems to be infeasible.
For a general $4 \times 4$ matrix, we aborted the computation after $1$ week.

Instead, we generate many pseudorandom $n \times n$ matrices $A$, identify the top-left entry of $\Sink(A)$ as an algebraic number for each, and set up systems of linear equations in $c_S(n)$.
We describe these three steps next.

For the first step, we generate matrices with entries from $\{1, 2, \dots, 20\}$.
Since the entries are integers, each $M(S)$ is also an integer; this will be important in the third step.
For each matrix, we check that none of its minors are $0$, since a $0$ minor implies $M(S) = 0$ for several $S$, removing the dependence on the corresponding coefficients $c_S(n)$.
If any minors are $0$, we discard that matrix.

In the second step, for each matrix~$A$ generated in the first step, we determine a polynomial equation satisfied by the top-left entry $x$ of $\Sink(A)$.
There are two possible methods.
One method is to use Gr\"obner bases;
this is faster than the Gr\"obner basis computation for a matrix with symbolic entries, but for $n \geq 5$ it is still slow.
Therefore we use another method, which is to apply the iterative scaling process to obtain a numeric approximation to $\Sink(A)$ and then guess a polynomial for its top-left entry.
We begin by numericizing the integer entries of $A$ to high precision.
For $n = 4$ we use precision $2^{12}$, and for $n = 5$ we use precision $2^{15}$.
Then we iteratively scale until we reach a fixed point.
The precision of the entries drops during the scaling process, but with the initial precisions $2^{12}$ and $2^{15}$ we get entries with sufficiently high precision that we can reliably recognize them using PSLQ.
The expected degree of the polynomial is $d = \binom{2 n - 2}{n - 1}$ according to Conjecture~\ref{polynomial form}.
Building in redundancy, we use Mathematica's \texttt{RootApproximant}~\cite{RootApproximant} to approximate $x$ by an algebraic number with target degree $d + 2$.
When the output has degree $d$, which is almost always the case, this is strong evidence that the approximation is in fact the exact algebraic number we seek.
Occasionally the output has degree less than $d$, in which case we discard the matrix; for example, one $4 \times 4$ matrix produced a polynomial with degree $8$ rather than $20$, presumably because the general degree-$20$ polynomial, when evaluated at the entries of the matrix, is reducible.
Finally, we perform a check on the output by computing the ratio of its leading coefficient to its constant coefficient.
This ratio should be $\frac{M(D(n, n))}{M(\{\})}$, assuming Conjecture~\ref{polynomial form} is correct and $c_{D(n, n)}(n) = 1$.
All outputs passed this test.
We record the matrix~$A$ along with the polynomial equation satisfied by $x$, and this will give us $1$ equation in the third step.
(In fact we can get $n^2$ equations by applying PSLQ to each entry of the numeric approximation to $\Sink(A)$ and recording each polynomial along with the matrix obtained by swapping the appropriate rows and columns of $A$.
For $n = 5$ this is worthwhile, since our implementation of the iterative scaling process takes roughly $3$ minutes to reach a fixed point with initial precision $2^{15}$.)

The third step is to determine the coefficients $c_S(n)$ in the coefficient of $x^k$ in Conjecture~\ref{polynomial form}.
For a given $k$, we do this by solving a system of linear equations involving the coefficients $c_S(n)$ where $\size{S} = k$.
We assume $c_T(n) = c_S(n)$ if $T \equiv S$, so it suffices to determine $c_S(n)$ for one representative $S$ of each equivalence class, analogous to Theorem~\ref{3 by 3 polynomial}.
This reduces the number of unknown coefficients, which reduces the number of equations we need, which reduces the number of matrices $A$ we apply the iterative scaling process to in the second step above.
However, first we must partition $\{S \subseteq D(n, n) : \size{S} = k\}$ into its equivalence classes under $\equiv$.
Some care must be taken to do this efficiently;
we make use of the fact that row permutations commute with column permutations, so it suffices to apply row permutations first.
Once we have computed the equivalence classes, we take each polynomial computed in the second step above, scale it by an integer so that its leading coefficient is $M(D(n, n))$ (and therefore its constant coefficient is $M(\{\})$), extract the coefficient of $x^k$, and set this coefficient equal to $\sum_S c_S(n) \Sigma(S)$ where the sum is over one representative from each equivalence class of size-$k$ subsets.
The number of equivalence classes tells us how many such equations we need in order to solve for the unknown coefficients $c_S(n)$.
We include more equations than necessary in the system, building in redundancy, so that if a solution is found then we can be confident that the conjectured form is correct.
Then we solve the system.
In practice, even setting up the system can be computationally expensive.
For $n = 5$ and $k = 4$, our initial implementation took $6$ days to set up the system (before solving!)\ because there are $1518$ unknown coefficients $c_S(n)$, so we need at least that many equations, and each equation involves $\binom{70}{4} = 916895$ monomials, each of which is a product of $\binom{8}{4} = 70$ determinants.
For small $k$, a more efficient way to construct each equation is to take advantage of the fact that, for each pair $S, T$ of subsets of $D(5, 5)$, the products $M(S)$ and $M(T)$ have most factors in common, and almost all are $\Gamma$ factors.
Therefore, we can reduce the number of determinant computations by dividing both sides by
$
	M(\{\})
	= \prod_{(R, C) \in D(5, 5)}
	\Gamma\!\left(\!\begin{smallmatrix}
		R\vphantom{\{} \\
		C\vphantom{\{}
	\end{smallmatrix}\!\right)
$,
which is nonzero by the assumption that no minor is $0$.
On the left, we compute $M(\{\})$ once and divide the extracted coefficient of $x^k$ by $M(\{\})$.
On the right, we divide each $M(S)$ by $M(\{\})$.
Instead of computing $M(S)/M(\{\})$ from definitions, we precompute the ratio
$
	\Delta\!\left(\!\begin{smallmatrix}
		R\vphantom{\{} \\
		C\vphantom{\{}
	\end{smallmatrix}\!\right)
	/
	\Gamma\!\left(\!\begin{smallmatrix}
		R\vphantom{\{} \\
		C\vphantom{\{}
	\end{smallmatrix}\!\right)
$
for each $(R, C) \in D(5, 5)$; then, for each $S$ with $\size{S} = k$, we use the precomputed ratios to compute
\[
	\frac{M(S)}{M(\{\})}
	=
	\prod_{(R, C) \in S}
	\frac{\Delta\!\left(\!\begin{smallmatrix}
		R\vphantom{\{} \\
		C\vphantom{\{}
	\end{smallmatrix}\!\right)}
	{\Gamma\!\left(\!\begin{smallmatrix}
		R\vphantom{\{} \\
		C\vphantom{\{}
	\end{smallmatrix}\!\right)}.
\]
This product can be computed quickly small since $k$ is small.
(This shortcut is also used by \textsc{SinkhornPolynomials}~\cite{SinkhornPolynomials} to compute polynomials more quickly when no minor is $0$.)

We now carry out these three steps for $n = 4$ to obtain a polynomial for the top-left entry $x$ of $\Sink(A)$ for $4 \times 4$ matrices $A$.
This polynomial is analogous to the polynomial in Theorem~\ref{3 by 3 polynomial} for $3 \times 3$ matrices.
The number of equivalence classes of size-$k$ subsets of $D(4, 4)$ for $k = 0, 1, \dots, 20$ is
\[
	1, 4, 12, 40, 123, 324, 724, 1352, 2108, 2760, 3024, 2760, \dots, 4, 1.
\]
The number of unknown coefficients $c_S(4)$ is the sum of these numbers, which is $17920$.
We use the definition $c_{\{\}}(4) = 1$ and the conjecture $c_{D(4, 4)}(4) = 1$, and we solve the remaining $19$ systems of linear equations, the largest of which requires $3024$ equations.
Unfortunately, for each $k \in \{4, 5, \dots, 16\}$ the system has multiple solutions.
For example, when we solve the system for $k = 4$, only $104$ of the $123$ unknown coefficients $c_S(4)$ with $\size{S} = 4$ are uniquely determined;
the other $19$ are parameterized by $2$ free variables.
For $k = 10$, the solution space has dimension $1141$.
This is not a weakness of the interpolation strategy but rather implies that the coefficients of $x^4, x^5, \dots, x^{16}$ have multiple representations as linear combinations of class sums $\Sigma(S)$ and therefore that relations exist among these class sums.
However, by setting the free variables to $0$ (or any other values), we obtain numeric coefficients $c_S(4)$.

\begin{conjecture}\label{4 by 4 polynomial - large}
There are $17920$ rational numbers
\begin{align*}
	c_{D(4, 4)}(4) &= 1 \\
	&\;\; \vdots \\
	c_{\subsetmatrix{
		\{\} \\
		\{\}
	}}\!(4) &= -4 \\
	c_{\subsetmatrix{
		\{2\} \\
		\{2\}
	}}\!(4) &= -2 \\
	c_{\subsetmatrix{
		\{2, 3\} \\
		\{2, 3\}
	}}\!(4) &= 0 \\
	c_{\subsetmatrix{
		\{2, 3, 4\} \\
		\{2, 3, 4\}
	}}\!(4) &= 2 \\
	c_{\{\}}(4) &= 1
\end{align*}
such that the top-left entry $x$ of the Sinkhorn limit of a $4 \times 4$ matrix satisfies
\[
	\sum_{k = 0}^{20} \left(\sum_S c_S(4) \Sigma(S)\right) x^k = 0
\]
where the inner sum is over one representative $S$ from each equivalence class of size-$k$ subsets of $D(4, 4)$.
\end{conjecture}

\begin{example}
Consider the matrix
\[
	A =
	\begin{bmatrix}
		3 & 1 & 2 & 2 \\
		2 & 2 & 2 & 1 \\
		1 & 2 & 3 & 2 \\
		1 & 4 & 2 & 3
	\end{bmatrix},
\]
and let $x$ be the top-left entry of $\Sink(A)$.
Conjecture~\ref{4 by 4 polynomial - large} gives
\small
\begin{multline*}
	382625520076800 x^{20} - 15753370260418560 x^{19} + 224644720812019200 x^{18} \\
	- 1949693785825830912 x^{17} + 11625683820163305984 x^{16} - 50547801347982259200 x^{15} \\
	+ 165827284134596798976 x^{14} - 419342005165888558080 x^{13} + 828111699533723747328 x^{12} \\
	- 1284220190788992755712 x^{11} + 1558933050581256001536 x^{10} - 1456458194243244008448 x^9 \\
	+ 999710159534823121920 x^8 - 435645828109071673344 x^7 + 31060141423020794880 x^6 \\
	+ 122853118332060905472 x^5 - 110123924197151416320 x^4 + 53612068706701295616 x^3 \\
	- 16383341182381572096 x^2 + 2975198930601246720 x - 246790694704250880
	= 0.
\end{multline*}
\normalsize
\end{example}

The values of the coefficients $c_S(4)$ in Conjecture~\ref{4 by 4 polynomial - large} are not all canonical, since we don't have natural conditions under which they are uniquely determined.
However, we continue under the assumption that there is a unique natural function $c_S(n)$ and seek to identify it.

In principle, we can use the same method to interpolate a polynomial equation for $n \times n$ matrices for any given $n$.
However, the computation is formidable.
For $5 \times 5$ matrices, we were able to compute the coefficients for $k = 4$ and $k = 5$ (where the families of representations are respectively $8$-dimensional and $44$-dimensional), but for $k = 6$ we could not get the computation of equivalence classes to finish (aborted after $2$ weeks).
Completing the computation up to the halfway point $k = \frac{1}{2} \binom{8}{4} = 35$, after which we could use the symmetry in Conjecture~\ref{symmetry}, seems to be infeasible.

For $6 \times 6$ matrices, PSLQ must recognize algebraic numbers with degree $\binom{10}{5} = 252$.
We couldn't get any of these computations to finish with sufficiently high precision.

\section{Rectangular matrices and combinatorial structure}\label{Rectangular matrices and combinatorial structure}

The coefficients $c_S(n)$ for $n \in \{2, 3, 4, 5\}$ that we have computed so far are not enough to guess the general formula for $c_S(n)$.
In this section, we expand the scope to include matrices that are not square.
This allows us to compute enough coefficients to identify formulas in special cases.
These formulas reveal the relevant combinatorial structure on subsets $S$, and this allows us to piece together a general formula for the coefficients, resulting in Conjecture~\ref{general polynomial}.

Doubly stochastic matrices are necessarily square.
This is because, in every matrix, the sum of the row sums is equal to the sum of the column sums.
Therefore we must generalize how we scale.

\begin{definition*}
Let $A$ be a positive $m \times n$ matrix.
The \emph{Sinkhorn limit} of $A$ is the matrix obtained by iteratively scaling so that each row sum is $1$ and each column sum is $\frac{m}{n}$.
Its existence was established (in a more general form) by Sinkhorn in a 1967 paper~\cite{Sinkhorn 1967}.
\end{definition*}

Conjecture~\ref{polynomial form} generalizes to $m \times n$ matrices as follows.
We extend the coefficients $c_S(n)$ in the previous section to coefficients $c_S(m, n)$, where $c_S(n, n) = c_S(n)$.
For each $m$ and $n$, we scale the coefficients so that $c_{\{\}}(m, n) = 1$.

\begin{conjecture}\label{general polynomial form}
Let $m \geq 1$ and $n \geq 1$.
There exist rational numbers $c_S(m, n)$, indexed by subsets $S \subseteq D(m, n)$, such that, for every positive $m \times n$ matrix~$A$, the top-left entry $x$ of $\Sink(A)$ satisfies
\[
	\sum_{S \subseteq D(m, n)} c_S(m, n) M(S) x^{\size{S}} = 0.
\]
\end{conjecture}

To interpolate values of $c_S(m, n)$, we use the method described in Section~\ref{Square matrices}.
Overall, we spent $1.5$~years of CPU time iteratively scaling matrices of various dimensions and recognizing $102000$ algebraic numbers.
An additional month of CPU time was spent setting up and solving systems of linear equations in the coefficients $c_S(m, n)$, which resulted in the identification of $63000$ rational coefficients (and an additional $56000$ coefficients parameterized by free variables).

Rather than fixing $m$ and $n$ and varying $S$ as in previous sections, we change our perspective now by fixing $S$ and working to identify $c_S(m, n)$ as a function of $m$ and $n$.
Accordingly, we define $c_S(m, n)$ to be $0$ if $S \nsubseteq D(m, n)$ (that is, $S$ contains row or column indices that are larger than an $m \times n$ matrix supports).
We begin with subsets $S$ where $\size{S} = 1$.

\begin{example}
Let $S = \subsetmatrix{ \{\} \\ \{\} }$.
Equation~\eqref{2 by 2 polynomial - compact general form} implies $c_S(2, 2) = -2$, Theorem~\ref{3 by 3 polynomial} implies $c_S(3, 3) = -3$, and Conjecture~\ref{4 by 4 polynomial - large} implies $c_S(4, 4) = -4$.
The values of $c_S(m, n)$ we computed for several additional matrix sizes appear in the following table.
\small
\[
	\begin{array}{r|rrrr}
		& n = 1 & 2 & 3 & 4 \\
		\hline
		m = 1 & -1 & -2 & -3 & -4 \\
		2 & -1 & -2 & -3 & -4 \\
		3 & -1 & -2 & -3 & -4 \\
		4 & -1 & -2 & -3 & -4
	\end{array}
\]
\normalsize
This suggests the formula $c_S(m, n) = -n$.
\end{example}

There is an asymmetry in our definition of $\Sink(A)$ for non-square matrices, since $\Sink(A)$ has row sums $1$ and column sums $\frac{m}{n}$.
However, we expect some symmetry in the coefficients $c_S(m, n)$ since $m \Sink(A\tr) = n \Sink(A)\tr$.
We can obtain this symmetry by considering $\frac{1}{m} \Sink(A)$, which has row sums $\frac{1}{m}$ and column sums $\frac{1}{n}$.
The form of the polynomial for the top-left entry $y$ of $\frac{1}{m} \Sink(A)$ can be obtained from Conjecture~\ref{general polynomial form} by substituting $x = m y$.
Therefore, the coefficients $c_S(m, n)$ should satisfy $m^{\size{S}} c_S(m, n) = n^{\size{S}} c_{S\tr}(n, m)$, where $S\tr$ is defined by
\[
	\left(\subsetmatrix{
		R_1 & R_2 & \cdots & R_k\vphantom{\{} \\
		C_1 & C_2 & \cdots & C_k\vphantom{\{}
	}\right)\tr
	=
	\subsetmatrix{
		C_1 & C_2 & \cdots & C_k\vphantom{\{} \\
		R_1 & R_2 & \cdots & R_k\vphantom{\{}
	}.
\]

\begin{example}
As in the previous example, let $S = \subsetmatrix{ \{\} \\ \{\} }$.
Since $S\tr = S$, we have $m c_S(m, n) = n c_S(n, m)$.
In other words, $m c_S(m, n)$ is a symmetric function of $m$ and $n$.
Indeed, the table of values suggests $m c_S(m, n) = -m n$.
\end{example}

The coefficients for other size-$1$ subsets $S$ also seem to be simple polynomial functions of $m$ and $n$.

\begin{example}
For $S = \subsetmatrix{ \{2\} \\ \{2\} }$, the data suggests $m c_S(m, n) = m + n - m n$ (for all $m \geq 2$ and $n \geq 2$).
For $S = \subsetmatrix{ \{2\} \\ \{3\} }$, the data suggests the same formula $m c_S(m, n) = m + n - m n$; indeed this is expected, because $\subsetmatrix{ \{2\} \\ \{2\} } \equiv \subsetmatrix{ \{2\} \\ \{3\} }$.
Since every size-$1$ subset
$
	\subsetmatrix{
		R \vphantom{\{} \\
		C \vphantom{\{}
	}
$ is equivalent to $S = \subsetmatrix{ \{2, 3, \dots, \size{R} + 1\} \\ \{2, 3, \dots, \size{R} + 1\} }$, it suffices to consider the latter.
For $S = \subsetmatrix{ \{2, 3\} \\ \{2, 3\} }$, the data suggests $m c_S(m, n) = 2 m + 2 n - m n$.
For $S = \subsetmatrix{ \{2, 3, 4\} \\ \{2, 3, 4\} }$, it suggests $m c_S(m, n) = 3 m + 3 n - m n$.
\end{example}

These formulas lead to the following.

\begin{conjecture}\label{coefficient for size-1 subsets}
Let $m \geq 1$ and $n \geq 1$, and let
$
	S =
	\subsetmatrix{
		R \vphantom{\{} \\
		C \vphantom{\{}
	}
	\subseteq D(m, n)
$.
Then $m c_S(m, n) = \size{R} (m + n) - m n$.
\end{conjecture}

Next we consider subsets $S$ where $\size{S} = 2$.

\begin{example}\label{type 1 link example}
Let $S = \subsetmatrix{ \{\} & \{2\} \\ \{\} & \{2\} }$.
Several values of $m^2 c_S(m, n)$ appear in the following table.
\small
\[
	\begin{array}{r|rrrr}
		& n = 2 & 3 & 4 & 5 \\
		\hline
		m = 2 & 4 & 12 & 24 & 40 \\
		3 & 12 & 36 & 72 & 120 \\
		4 & 24 & 72 & 144 & 240 \\
		5 & 40 & 120 & 240 & 400
	\end{array}
\]
\normalsize
This suggests $m^2 c_S(m, n) = (m - m n) (n - m n)$.
\end{example}

\begin{example}\label{type 2 link example}
Let $S = \subsetmatrix{ \{2\} & \{2\} \\ \{2\} & \{3\} }$.
The two minor specifications in $S$ involve the same row but different columns.
Therefore there is no subset $T$ in the equivalence class of $S$ such that $T\tr = T$, so we do not expect $m^2 c_S(m, n)$ to be a symmetric function of $m$ and $n$.
Here are several of its values (where the bottom-right entry is not included because we have no data for $6 \times 6$ matrices):
\small
\[
	\begin{array}{r|rrrr}
		& n = 3 & 4 & 5 & 6 \\
		\hline
		m = 2 & -3 & 0 & 5 & 12 \\
		3 & 0 & 16 & 40 & 72 \\
		4 & 9 & 48 & 105 & 180 \\
		5 & 24 & 96 & 200 & 336 \\
		6 & 45 & 160 & 325
	\end{array}
\]
\normalsize
This suggests $m^2 c_S(m, n) = (n - m n) (2 m + n - m n)$.
\end{example}

By interpolating polynomial formulas for additional subsets $S$, one conjectures that the scaled coefficient $m^{\size{S}} c_S(m, n)$ is a polynomial function of $m$ and $n$ with degree $\size{S}$ in each variable.
Moreover, there seem to be two basic relationships between minor specifications in $S$ that play a central role in the form of the polynomial function.
These are illustrated by the previous two examples.

\begin{definition*}
We say that two minor specifications $(R_1, C_1), (R_2, C_2) \in D(m, n)$ are \emph{linked} if either of the following conditions holds.
\begin{enumerate}
\item
Their sizes differ by $1$, and the smaller is a subset of the larger.
That is,
\begin{itemize}
\item
$R_1 = R_2 \setminus \{s\}$ and $C_1 = C_2 \setminus \{t\}$ for some $s \in R_2$ and $t \in C_2$, or
\item
$R_2 = R_1 \setminus \{s\}$ and $C_2 = C_1 \setminus \{t\}$ for some $s \in R_1$ and $t \in C_1$.
\end{itemize}
In this case, we say that they form a \emph{type-1 link}.
\item
Their sizes are the same, and they differ in exactly $1$ row index or $1$ column index.
That is,
\begin{itemize}
\item
$R_1 = R_2$ and $C_1 \setminus \{s\} = C_2 \setminus \{t\}$ for some $s \in C_1$ and $t \in C_2$ such that $s \neq t$, or
\item
$C_1 = C_2$ and $R_1 \setminus \{s\} = R_2 \setminus \{t\}$ for some $s \in R_1$ and $t \in R_2$ such that $s \neq t$.
\end{itemize}
In this case, we say that they form a \emph{type-2 link}.
\end{enumerate}
\end{definition*}

The minor specifications $(\{\}, \{\})$ and $(\{2\}, \{2\})$ in Example~\ref{type 1 link example} form a type-1 link, and the minor specifications $(\{2\}, \{2\})$ and $(\{2\}, \{3\})$ in Example~\ref{type 2 link example} form a type-2 link.
Additional size-$2$ subsets $S$ are illustrated in Table~\ref{links}, along with corresponding formulas for $m^2 c_S(m, n)$ that agree with numeric data.

\begin{table}
\[
	\begin{array}{cccc}
		S & \text{diagram} & \text{link type} & m^2 c_S(m, n) \\
		\hline
		\hline
		\subsetmatrix{
			\{2\} & \{2, 3\} \\
			\{2\} & \{2, 3\}
		}
		& \vcenter{\hbox{\includegraphics{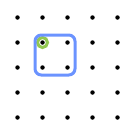}}}
		& 1
		& (m + 2 n - m n) (2 m + n - m n)
		\\
		\subsetmatrix{
			\{2, 3\} & \{2, 3, 4\} \\
			\{2, 3\} & \{2, 3, 4\}
		}
		& \vcenter{\hbox{\includegraphics{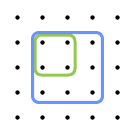}}}
		& 1
		& (2 m + 3 n - m n) (3 m + 2 n - m n)
		\\
		\subsetmatrix{
			\{2, 3, 4\} & \{2, 3, 4, 5\} \\
			\{2, 3, 4\} & \{2, 3, 4, 5\}
		}
		& \vcenter{\hbox{\includegraphics{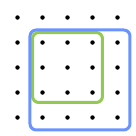}}}
		& 1
		& (3 m + 4 n - m n) (4 m + 3 n - m n)
		\\
		\hline
		\subsetmatrix{
			\{2\} & \{3\} \\
			\{2\} & \{2\}
		}
		& \vcenter{\hbox{\includegraphics{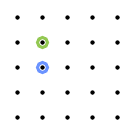}}}
		& 2
		& (m - m n) (m + 2 n - m n)
		\\
		\subsetmatrix{
			\{2, 3\} & \{3, 4\} \\
			\{2, 3\} & \{2, 3\}
		}
		& \vcenter{\hbox{\includegraphics{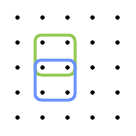}}}
		& 2
		& (2 m + n - m n) (2 m + 3 n - m n)
		\\
		\subsetmatrix{
			\{2, 3, 4\} & \{3, 4, 5\} \\
			\{2, 3, 4\} & \{2, 3, 4\}
		}
		& \vcenter{\hbox{\includegraphics{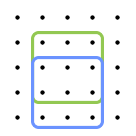}}}
		& 2
		& (3 m + 2 n - m n) (3 m + 4 n - m n)
		\\
		\hline
		\subsetmatrix{
			\{2\} & \{2, 3, 4\} \\
			\{2\} & \{2, 3, 4\}
		}
		& \vcenter{\hbox{\includegraphics{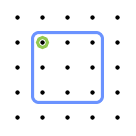}}}
		& \text{not linked}
		& (m + n - m n) (3 m + 3 n - m n)
		\\
		\subsetmatrix{
			\{2, 3\} & \{4, 5\} \\
			\{2, 3\} & \{2, 3\}
		}
		& \vcenter{\hbox{\includegraphics{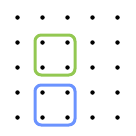}}}
		& \text{not linked}
		& (2 m + 2 n - m n)^2
	\end{array}
\]
\caption{Examples of size-$2$ subsets $S$, with formulas for $m^2 c_S(m, n)$ that agree with numeric data. Each minor specification in these examples consists of a contiguous block of rows and a contiguous block of columns, so we can draw them as squares (drawn here on $5 \times 5$ matrices).}
\label{links}
\end{table}

These examples, along with others, suggest general formulas for size-$2$ subsets.

\begin{conjecture}\label{coefficient for size-2 subsets}
Let $m \geq 1$ and $n \geq 1$, and let
$
	S =
	\subsetmatrix{
		R_1 & R_2\vphantom{\{} \\
		C_1 & C_2\vphantom{\{}
	}
	\subseteq D(m, n)
$.
\begin{itemize}
\item
If $(R_1, C_1)$ and $(R_2, C_2)$ form a type-1 link with $\size{R_1} + 1 = \size{R_2}$, then
\[
	m^2 c_S(m, n)
	= \big(\size{R_1} (m + n) + n - m n\big) \big(\size{R_1} (m + n) + m - m n\big).
\]
\item
If $(R_1, C_1)$ and $(R_2, C_2)$ form a type-2 link with $R_1 = R_2$, then
\[
	m^2 c_S(m, n)
	= \big(\size{R_1} (m + n) - m - m n\big) \big(\size{R_1} (m + n) + m - m n\big).
\]
\item
If $(R_1, C_1)$ and $(R_2, C_2)$ form a type-2 link with $C_1 = C_2$, then
\[
	m^2 c_S(m, n)
	= \big(\size{R_1} (m + n) - n - m n\big) \big(\size{R_1} (m + n) + n - m n\big).
\]
\item
If $(R_1, C_1)$ and $(R_2, C_2)$ are not linked, then
\[
	m^2 c_S(m, n)
	= \big(\size{R_1} (m + n) - m n\big) \big(\size{R_2} (m + n) - m n\big).
\]
\end{itemize}
\end{conjecture}

Conjectures~\ref{coefficient for size-1 subsets} and \ref{coefficient for size-2 subsets} imply that, if $\size{S} \leq 2$, then the polynomial $m^{\size{S}} c_S(m, n)$ is a product of factors that are linear in $m$ and linear in $n$.
When we consider subsets $S$ with $\size{S} \geq 3$, we find additional polynomials with this property.
In particular, when $S$ has no linked pairs, there seems to be a simple description of $m^{\size{S}} c_S(m, n)$ as follows.

\begin{conjecture}\label{disconnected}
Let $m \geq 1$ and $n \geq 1$, and let
$
	S =
	\subsetmatrix{
		R_1 & R_2 & \cdots & R_k\vphantom{\{} \\
		C_1 & C_2 & \cdots & C_k\vphantom{\{}
	}
	\subseteq D(m, n)
$.
If $S$ contains no linked pairs, then
\[
	m^k c_S(m, n)
	= \prod_{i = 1}^k \big(\size{R_i} (m + n) - m n\big).
\]
\end{conjecture}

Conjecture~\ref{disconnected} suggests more generally that the structure of the polynomial $m^k c_S(m, n)$ is determined by the connected components of linked pairs in $S$ and, moreover, that the contributions of the components are independent of each other.
This is the content of Conjecture~\ref{connected components} below.
We make this precise by defining the following graph.

\begin{notation*}
For each $S \subseteq D(m, n)$, let $G_S$ be the graph whose vertex set is $S$ and whose edges connect pairs of linked vertices.
\end{notation*}

\begin{example}
Let $S = \subsetmatrix{ \{\} & \{2\} & \{2, 3\} & \{2, 4\} \\ \{\} & \{2\} & \{3, 4\} & \{3, 4\} }$.
The first two elements of $S$ form a type-1 link, and the last two form a type-2 link.
These are the only links, so $G_S$ is the graph on $4$ vertices with two non-adjacent edges.
The two connected components are equivalent to the subsets $\subsetmatrix{ \{\} & \{2\} \\ \{\} & \{2\} }$ in Example~\ref{type 1 link example} and $\subsetmatrix{ \{2, 3\} & \{3, 4\} \\ \{2, 3\} & \{2, 3\} }$ in Table~\ref{links}, respectively.
The values we computed of $m^4 c_S(m, n)$ are as follows.
\small
\[
	\begin{array}{r|rrrr}
		& n = 4 & 5 & 6 & 7 \\
		\hline
		m = 4 & -2304 & -5040 & -7200 & -6552 \\
		5 & -2880 & 0 \\
		6 & 0 \\
		7 & 10080
	\end{array}
\]
\normalsize
This is consistent with $m^4 c_S(m, n) = (m - m n) (n - m n) (2 m + n - m n) (2 m + 3 n - m n)$, which is the product of the formulas in Example~\ref{type 1 link example} and Table~\ref{links}.
\end{example}

\begin{conjecture}\label{connected components}
Let $m \geq 1$ and $n \geq 1$.
For every $S \subseteq D(m, n)$,
\[
	m^{\size{S}} c_S(m, n) = \prod_T m^{\size{T}} c_T(m, n)
\]
where the product is over the connected components $T$ of $S$.
Moreover, since $\size{S} = \sum_T \size{T}$, this implies $c_S(m, n) = \prod_T c_T(m, n)$.
\end{conjecture}

Conjecture~\ref{connected components} is supported by all the rational coefficients we computed.
Assuming it is true, it suffices to determine $m^{\size{S}} c_S(m, n)$ for subsets $S$ consisting of a single connected component.

When a connected component consists of more than one linked pair, the corresponding polynomial is not necessarily a product of linear factors.

\begin{example}\label{component example}
Let $S = \subsetmatrix{ \{2\} & \{3\} & \{2, 3\} \\ \{2\} & \{3\} & \{2, 3\} }$.
Here is the diagram for $S$ on a $5 \times 5$ matrix:
\[
	\vcenter{\hbox{\includegraphics{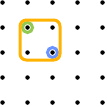}}}
\]
The minor specifications $(\{2\}, \{2\})$ and $(\{3\}, \{3\})$ are not linked, but each of the other two pairs is.
Therefore $G_S$ consists of a single connected component.
Here are several values of $m^3 c_S(m, n)$:
\small
\[
	\begin{array}{r|rrrr}
		& n = 3 & 4 & 5 & 6 \\
		\hline
		m = 3 & -27 & -70 & -161 & -324 \\
		4 & -70 & -256 & -682 & -1456 \\
		5 & -161 & -682 & -1875 & -4028 \\
		6 & -324 & -1456 & -4028
	\end{array}
\]
\normalsize
Without the value for $6 \times 6$ matrices, we cannot interpolate a cubic polynomial in $m$ and $n$.
However, by interpolating cubic polynomials in $n$ for the first three rows, we see that $m + n - m n$ is likely a factor.
Dividing each value in the table by this factor, we then interpolate a quadratic polynomial to obtain $m^3 c_S(m, n) = (m + n - m n) \left(2 m^2 + 2 n^2 + 6 m n - 3 m^2 n - 3 m n^2 + m^2 n^2\right)$.
The quadratic factor is irreducible.
\end{example}

An obvious question is whether there is a better way to write such polynomials, so that we can see the general structure.
In fact there is, using determinants.
Since the determinant of a block diagonal matrix is the product of the determinants of the blocks, determinant formulas are good candidates for functions that decompose as products over connected components.
In this direction, we next rewrite Conjecture~\ref{coefficient for size-2 subsets} using determinant formulas.

\begin{conjecture}\label{coefficient for size-2 subsets - determinants}
Let $m \geq 1$ and $n \geq 1$, and let
$
	S =
	\subsetmatrix{
		R_1 & R_2\vphantom{\{} \\
		C_1 & C_2\vphantom{\{}
	}
	\subseteq D(m, n)
$.
\begin{itemize}
\item
If $(R_1, C_1)$ and $(R_2, C_2)$ form a type-1 link with $\size{R_1} + 1 = \size{R_2}$, then
\[
	m^2 c_S(m, n)
	= \det \begin{bmatrix}
		\size{R_1} (m + n) - m n & m \\
		-n & \size{R_2} (m + n) - m n
	\end{bmatrix}.
\]
\item
If $(R_1, C_1)$ and $(R_2, C_2)$ form a type-2 link with $R_1 = R_2$, then
\[
	m^2 c_S(m, n)
	= \det \begin{bmatrix}
		\size{R_1} (m + n) - m n & m \\
		m & \size{R_2} (m + n) - m n
	\end{bmatrix}.
\]
\item
If $(R_1, C_1)$ and $(R_2, C_2)$ form a type-2 link with $C_1 = C_2$, then
\[
	m^2 c_S(m, n)
	= \det \begin{bmatrix}
		\size{R_1} (m + n) - m n & n \\
		n & \size{R_2} (m + n) - m n
	\end{bmatrix}.
\]
\item
If $(R_1, C_1)$ and $(R_2, C_2)$ are not linked, then
\[
	m^2 c_S(m, n)
	= \det \begin{bmatrix}
		\size{R_1} (m + n) - m n & 0 \\
		0 & \size{R_2} (m + n) - m n
	\end{bmatrix}.
\]
\end{itemize}
\end{conjecture}

For size-$1$ subsets
$
	S =
	\subsetmatrix{
		R \vphantom{\{} \\
		C \vphantom{\{}
	}
	\subseteq D(m, n)
$, we can rewrite Conjecture~\ref{coefficient for size-1 subsets} as the $1 \times 1$ determinant $m c_S(m, n) = \det \begin{bmatrix} \size{R} (m + n) - m n \end{bmatrix}$.
For the size-$0$ subset $S = \{\}$, the definition $c_{\{\}}(m, n) = 1$ is consistent with $c_{\{\}}(m, n)$ being the determinant of the $0 \times 0$ matrix.

Along with Conjecture~\ref{connected components}, this suggests a possible determinant formula for $m^{\size{S}} c_S(m, n)$ for an arbitrary subset $S$.
Since the matrix in this determinant formula resembles an adjacency matrix, we introduce the following notation.

\begin{notation*}
For each
$
	S =
	\subsetmatrix{
		R_1 & R_2 & \cdots & R_k\vphantom{\{} \\
		C_1 & C_2 & \cdots & C_k\vphantom{\{}
	}
$, define $\adj'_S(m, n)$ to be the $k \times k$ matrix with the property that, for all $i, j$ satisfying $1 \leq i < j \leq k$, the $2 \times 2$ submatrix $(\adj'_S(m, n))_{\{i, j\}, \{i, j\}}$ is the matrix in Conjecture~\ref{coefficient for size-2 subsets - determinants} for the subset
$
	\subsetmatrix{
		R_i & R_j\vphantom{\{} \\
		C_i & C_j\vphantom{\{}
	}
$, with the following additional condition.
We are building a matrix from a set $S$, and the determinant of this matrix should not depend on the order in which we write the elements of $S$.
Three of the cases in Conjecture~\ref{coefficient for size-2 subsets - determinants} are symmetric in $(R_i, C_i)$ and $(R_j, C_j)$, but the first case is not.
Therefore, if $(R_i, C_i)$ and $(R_j, C_j)$ form a type-1 link with $\size{R_i} = \size{R_j} + 1$, we permute the rows and columns relative to the first case in Conjecture~\ref{coefficient for size-2 subsets - determinants}, so that $(\adj'_S(m, n))_{\{i, j\}, \{i, j\}}$ is
\[
	\begin{bmatrix}
		\size{R_i} (m + n) - m n & -n \\
		m & \size{R_j} (m + n) - m n
	\end{bmatrix}.
\]
\end{notation*}

\begin{example}
As in Example~\ref{component example}, let $S = \subsetmatrix{ \{2\} & \{3\} & \{2, 3\} \\ \{2\} & \{3\} & \{2, 3\} }$.
We have
\[
	\adj'_S(m, n)
	= \begin{bmatrix}
		m + n - m n & 0 & m \\
		0 & m + n - m n & m \\
		-n & -n & 2 m + 2 n - m n
	\end{bmatrix}.
\]
Indeed, $\det \adj'_S(m, n)$ is equivalent to the formula in Example~\ref{component example} for $m^3 c_S(m, n)$.
\end{example}

In defining $\adj'_S(m, n)$, we inadvertently made choices about the signs of the off-diagonal entries.
Namely, negating the off-diagonal entries of the $2 \times 2$ matrices in Conjecture~\ref{coefficient for size-2 subsets - determinants} doesn't change their determinants, but of course negating a pair of off-diagonal entries in a larger matrix can change its determinant.
Therefore there is no reason to expect $\det \adj'_S(m, n)$ to give the intended coefficient for all subsets $S$, and in fact it does not.

\begin{example}
Let $S = \subsetmatrix{ \{2\} & \{2\} & \{2, 3\} \\ \{2\} & \{3\} & \{2, 3\} }$.
The minor specifications $(\{2\}, \{2\})$ and $(\{2\}, \{3\})$ form a type-2 link, and each forms a type-1 link with $(\{2, 3\}, \{2, 3\})$.
We have
\[
	\adj'_S(m, n)
	= \begin{bmatrix}
		m + n - m n & m & m \\
		m & m + n - m n & m \\
		-n & -n & 2 m + 2 n - m n
	\end{bmatrix}.
\]
The formula $\det \adj'_S(m, n)$ does not produce the following values we computed for $m^3 c_S(m, n)$.
\small
\[
	\begin{array}{r|rrrr}
		& n = 3 & 4 & 5 & 6 \\
		\hline
		m = 3 & 0 & -16 & -80 & -216 \\
		4 & -6 & -128 & -490 & -1200 \\
		5 & -36 & -432 & -1500 & -3528 \\
		6 & -108 & -1024 & -3380
	\end{array}
\]
\normalsize
However, if we alter the signs of the off-diagonal entries, we can get a determinant formula that produces these values, namely
\[
	m^3 c_S(m, n)
	= \det \begin{bmatrix}
		m + n - m n & m & m \\
		m & m + n - m n & -m \\
		-n & n & 2 m + 2 n - m n
	\end{bmatrix}.
\]
\end{example}

Further examples suggest that the signs of the off-diagonal entries in $\adj'_S(m, n)$ can always be altered in such a way that its determinant gives the value of $m^{\size{S}} c_S(m, n)$.
However, the graph $G_S$ does not determine the signs, as the following example shows.

\begin{example}\label{5-cycle}
For the subset $S = \subsetmatrix{ \{2, 3, 4\} & \{2, 3, 4\} & \{2, 3, 4\} & \{2, 3, 4\} & \{2, 3, 4\} \\ \{2, 3, 4\} & \{2, 3, 5\} & \{2, 5, 7\} & \{2, 6, 7\} & \{2, 4, 6\} }$, the graph $G_S$ is a $5$-cycle.
The determinant of
\small
\[
	\begin{bmatrix}
		3 m + 3 n - m n & m & 0 & 0 & -m \\
		m & 3 m + 3 n - m n & -m & 0 & 0 \\
		0 & -m & 3 m + 3 n - m n & m & 0 \\
		0 & 0 & m & 3 m + 3 n - m n & -m \\
		-m & 0 & 0 & -m & 3 m + 3 n - m n
	\end{bmatrix}
\]
\normalsize
gives the values of $m^5 c_S(m, n)$.
For the subset $T = \subsetmatrix{ \{2, 3, 4\} & \{2, 3, 4\} & \{2, 3, 4\} & \{2, 3, 4\} & \{2, 3, 4\} \\ \{2, 3, 4\} & \{2, 3, 5\} & \{3, 5, 6\} & \{4, 5, 6\} & \{2, 4, 6\} }$, the graph $G_T$ is also a $5$-cycle, but we need to alter one pair of signs; the determinant of
\small
\[
	\begin{bmatrix}
		3 m + 3 n - m n & m & 0 & 0 & -m \\
		m & 3 m + 3 n - m n & +m & 0 & 0 \\
		0 & +m & 3 m + 3 n - m n & m & 0 \\
		0 & 0 & m & 3 m + 3 n - m n & -m \\
		-m & 0 & 0 & -m & 3 m + 3 n - m n
	\end{bmatrix}
\]
\normalsize
gives the values of $m^5 c_T(m, n)$.
\end{example}

To determine the signs of the off-diagonal entries in general, we used ChatGPT o3-mini-high~\cite{ChatGPT}.
We presented the entire framework of the question, including our conjectures about the determinant formula for the coefficients, along with the previous example of two subsets with the same graph that require two different sign assignments.\footnote{For the conversation, see \url{https://chatgpt.com/share/67aba51d-c0f0-800d-a9cb-d3533aadc17c}.}
ChatGPT's responses were self-contradictory (it claimed that the product around cycles must be positive) and not entirely prescriptive, but it suggested that, for a linked pair of minor specifications $(R_i, C_i)$ and $(R_j, C_j)$, the relative positions of the row and column indices are relevant.
Namely, the signs of the corresponding off-diagonal entries should reflect the permutations mapping $R_i$ to $R_j$ and $C_i$ to $C_j$ as ordered lists.
We hadn’t considered this ourselves since the row and column orders play no role up to this point.
This was enough direction for us to find rules for the signs that agreed with all numeric data we had computed.
For a type-1 link, the index sets have different sizes, so to obtain a permutation we append the missing element to the smaller list.
For a type-2 link, the index sets have the same sizes but can differ in two elements, so to obtain a permutation we require that those two elements get mapped to each other.

\begin{example}
The subsets $S$ and $T$ in Example~\ref{5-cycle} each consist of $5$ type-2 links.
For $S$, we have $C_1 = \{2, 3, 4\}$ to $C_2 = \{2, 3, 5\}$.
The element $4 \in C_1$ plays the role of $5 \in C_2$, so we consider the permutation mapping $(2, 3, s)$ to $(2, 3, s)$.
This is the identity permutation, which has sign $1$, and indeed the $(1, 2)$ and $(2, 1)$ entries in the first matrix in Example~\ref{5-cycle} have sign $1$.
Analogously, the signs in the $5$-cycle
\[
	\{2, 3, 4\} \mapsto
	\{2, 3, 5\} \mapsto
	\{2, 5, 7\} \mapsto
	\{2, 6, 7\} \mapsto
	\{2, 4, 6\} \mapsto
	\{2, 3, 4\}
\]
are $1, -1, 1, -1, -1$.
These agree with the signs of the remaining off-diagonal entries.
On the other hand, for the subset $T$, the signs in the $5$-cycle
\[
	\{2, 3, 4\} \mapsto
	\{2, 3, 5\} \mapsto
	\{3, 5, 6\} \mapsto
	\{4, 5, 6\} \mapsto
	\{2, 4, 6\} \mapsto
	\{2, 3, 4\}
\]
are $1, 1, 1, -1, -1$.
These agree with the signs in the second matrix in Example~\ref{5-cycle}.
\end{example}

The sign of each permutation can be encoded more simply by the position where the differing index $s$ appears in each ordered list.
This formulation gives the definition of $\adj_S(m, n)$ in Section~\ref{Introduction} and results in Conjecture~\ref{general polynomial}.

We conclude this section by specializing Conjecture~\ref{general polynomial} to $2 \times n$ matrices.
For each exponent $k$, there are at most $2$ equivalence classes of subsets --- an equivalence class containing subsets $S$ such that $(\{\}, \{\}) \in S$ and another containing the rest.
Evaluating the two determinant formulas gives the following.

\begin{conjecture}
Let $n \geq 1$.
Let $S_k = \subsetmatrix{ \{\} & \{2\} & \{2\} & \cdots & \{2\} \\ \{\} & \{2\} & \{3\} & \cdots & \{k\} }$ for each $k \in \{1, \dots, n - 1, n\}$, let $T_k = \subsetmatrix{ \{2\} & \{2\} & \{2\} & \cdots & \{2\} \\ \{2\} & \{3\} & \{4\} & \cdots & \{k + 1\} }$ for each $k \in \{0, 1, \dots, n - 1\}$, and define the associated coefficients by
\begin{align*}
	2^k c_{S_k}(2, n) &=
	\begin{cases}
		(-n)^{k - 1} (2 k - 2 n - 2)	& \text{if $1 \leq k \leq n$} \\
		0					& \text{if $k = 0$}
	\end{cases} \\
	2^k c_{T_k}(2, n) &=
	\begin{cases}
		(-n)^{k - 1} (2 k - n)	& \text{if $0 \leq k \leq n - 1$} \\
		0				& \text{if $k = n$.}
	\end{cases}
\end{align*}
(We have defined $c_{S_0}(2, n) = 0$ and $c_{T_n}(2, n) = 0$ despite $S_0$ and $T_n$ being undefined; this allows us to write the following sum simply.)
For every positive $2 \times n$ matrix~$A$, the top-left entry $x$ of $\Sink(A)$ satisfies
\[
	\sum_{k = 0}^n \big(
		c_{S_k}(2, n) \Sigma(S_k)
		+
		c_{T_k}(2, n) \Sigma(T_k)
	\big) x^k
	= 0.
\]
\end{conjecture}

Gr\"obner basis computations are feasible for $2 \times n$ matrices and establish that the previous conjecture is true for $n \leq 12$.

\section{Open questions}\label{Open questions}

We mention several open questions.
The main question is how to prove Conjecture~\ref{general polynomial}.
Until we understand why the two types of links arise, a simple proof seems unlikely.
Alternatively, toward an algebraic proof, one could interpolate polynomial equations satisfied by the entries of diagonal matrices $R$ and $C$, as we did for the entries of $\Sink(A)$, and hope to prove that all three matrices are correct by checking that they satisfy $\Sink(A) = R A C$.
This is how Nathanson established the Sinkhorn limit of a $2 \times 2$ matrix~\cite[Theorem~3]{Nathanson 2x2}.
However, this works for $2 \times 2$ matrices because we have explicit expressions for the entries of $\Sink(A)$, rather than specifications as roots.
For $3 \times 3$ matrices, the diagonal entries of $R$ and $C$ are roots of degree-$6$ polynomials, and generically the product of two such roots has degree $6^2$;
this is too big, since the entries of $\Sink(A)$ have degree $6$.
Therefore we would need to find a single degree-$6$ field extension that contains all entries of $\Sink(A)$, $R$, and $C$ so that we can perform arithmetic on them symbolically.
What is this extension?

Second, as we saw in Section~\ref{3 by 3}, the lack of a unique representation of the coefficient of $x^3$ for $3 \times 3$ matrices is due to a relation among $12$ monomials $M(S)$ with $\size{S} = 3$.
What is the combinatorial structure of such relations?
Similarly, is there structure in relations among $\Sigma(S)$?
For example, the coefficient of $x^5$ for $3 \times 4$ matrices has a $1$-dimensional family of representations due to the relation
\small
\begin{multline*}
	\Sigma\!\left(\!\subsetmatrix{
		\{\} & \{2\} & \{2\} & \{3\} & \{2, 3\} \\
		\{\} & \{2\} & \{3\} & \{4\} & \{2, 4\}
	}\!\right)
	+ \Sigma\!\left(\!\subsetmatrix{
		\{\} & \{2\} & \{2\} & \{2, 3\} & \{2, 3\} \\
		\{\} & \{2\} & \{3\} & \{2, 3\} & \{2, 4\}
	}\!\right)
	+ \Sigma\!\left(\!\subsetmatrix{
		\{2\} & \{2\} & \{3\} & \{3\} & \{2, 3\} \\
		\{2\} & \{3\} & \{2\} & \{4\} & \{2, 3\}
	}\!\right) \\
	+ 2 \Sigma\!\left(\!\subsetmatrix{
		\{\} & \{2\} & \{3\} & \{2, 3\} & \{2, 3\} \\
		\{\} & \{2\} & \{3\} & \{2, 4\} & \{3, 4\}
	}\!\right)
	+ 2 \Sigma\!\left(\!\subsetmatrix{
		\{2\} & \{2\} & \{2\} & \{3\} & \{2, 3\} \\
		\{2\} & \{3\} & \{4\} & \{2\} & \{3, 4\}
	}\!\right)
	+ 2 \Sigma\!\left(\!\subsetmatrix{
		\{2\} & \{2\} & \{3\} & \{2, 3\} & \{2, 3\} \\
		\{2\} & \{3\} & \{4\} & \{2, 4\} & \{3, 4\}
	}\!\right) \\
	=
	\Sigma\!\left(\!\subsetmatrix{
		\{\} & \{2\} & \{3\} & \{2, 3\} & \{2, 3\} \\
		\{\} & \{2\} & \{3\} & \{2, 3\} & \{2, 4\}
	}\!\right)
	+ \Sigma\!\left(\!\subsetmatrix{
		\{2\} & \{2\} & \{2\} & \{3\} & \{2, 3\} \\
		\{2\} & \{3\} & \{4\} & \{2\} & \{2, 3\}
	}\!\right)
	+ \Sigma\!\left(\!\subsetmatrix{
		\{2\} & \{2\} & \{3\} & \{2, 3\} & \{2, 3\} \\
		\{2\} & \{3\} & \{4\} & \{2, 3\} & \{2, 4\}
	}\!\right) \\
	+ 2 \Sigma\!\left(\!\subsetmatrix{
		\{\} & \{2\} & \{2\} & \{3\} & \{2, 3\} \\
		\{\} & \{2\} & \{3\} & \{4\} & \{2, 3\}
	}\!\right)
	+ 2 \Sigma\!\left(\!\subsetmatrix{
		\{\} & \{2\} & \{2\} & \{2, 3\} & \{2, 3\} \\
		\{\} & \{2\} & \{3\} & \{2, 4\} & \{3, 4\}
	}\!\right)
	+ 2 \Sigma\!\left(\!\subsetmatrix{
		\{2\} & \{2\} & \{3\} & \{3\} & \{2, 3\} \\
		\{2\} & \{3\} & \{2\} & \{4\} & \{3, 4\}
	}\!\right).
\end{multline*}
\normalsize
It would be interesting to understand these better.

Third, how does Corollary~\ref{3 by 3 matrix with row multiples} generalize to $m \times n$ matrices with linear dependencies among their rows or columns?
The coefficients in Corollary~\ref{3 by 3 matrix with row multiples} have unique representations as linear combinations of class sums $\Sigma(S)$ where $S \subseteq D(2, 3)$, namely
\begin{align*}
	e_3 &=
		\Sigma\!\left(\!\subsetmatrix{
			\{\} & \{2\} & \{2\} \\
			\{\} & \{2\} & \{3\}
		}\!\right) \\
	e_2 &=
		-2 \Sigma\!\left(\!\subsetmatrix{
			\{\} & \{2\} \\
			\{\} & \{2\}
		}\!\right)
		+ \Sigma\!\left(\!\subsetmatrix{
			\{2\} & \{2\} \\
			\{2\} & \{3\}
		}\!\right) \\
	e_1 &=
		3 \Sigma\!\left(\!\subsetmatrix{
			\{\} \\
			\{\}
		}\!\right) \\
	e_0 &=
		-\Sigma\!\left(\!\subsetmatrix{
			\vphantom{\{} \\
			\vphantom{\{}
		}\!\right).
\end{align*}
What is the general formula?

Fourth, why does Conjecture~\ref{general polynomial} imply Conjectures~\ref{symmetry} and \ref{binomial sum}?

Finally, there is a generalization of Conjecture~\ref{general polynomial} for which the determinant formula is not known.
Decades before Sinkhorn's paper, the iterative scaling process was introduced by Kruithof~\cite{Kruithof} in the context of predicting telephone traffic.
In this application, rather than scaling to obtain row and column sums of $1$, each row and column has a potentially different target sum.
Sinkhorn~\cite{Sinkhorn 1967} showed that the limit exists.
We call this limit the \emph{Kruithof limit}.

\begin{example}
Kruithof~\cite[Appendix~3d]{Kruithof} considered the matrix
\[
	A =
	\begin{bmatrix}
		2000 & 1030 & 650 & 320 \\
		1080 & 1110 & 555 & 255 \\
		720 & 580 & 500 & 200 \\
		350 & 280 & 210 & 160
	\end{bmatrix}
\]
with target row sums $V = \begin{bmatrix} 6000 & 4000 & 2500 & 1000 \end{bmatrix}\tr$ and target column sums $W = \begin{bmatrix} 6225 & 4000 & 2340 & 935 \end{bmatrix}$.
Let $x$ be the top-left entry of the Kruithof limit.
Numerically, $x \approx 3246.38700234$.
A Gr\"obner basis computation gives the equation
\small
\begin{align*}
	62&11170485642866385308015185014605806684592592997303612 x^{20} \\
	&- 1911288675240357642608985257264441863326549355081446688219995 x^{19} \\
	&+ \cdots \\
	&- 980316295756763597938629190043558577216660563425441394040234375 \cdot 10^{72} x \\
	&+ 60077293526471262201893650291744622440239260152893558984375 \cdot 10^{79}
	= 0
\end{align*}
\normalsize
satisfied by $x$.
In particular, $x$ has degree $20$.
\end{example}

The previous example suggests that the degrees of entries of Kruithof limits are the same as those for Sinkhorn limits with the same dimensions.

\begin{conjecture}\label{Kruithof limit degree}
Let $m \geq 1$ and $n \geq 1$.
Let $A$ be a positive $m \times n$ matrix, let $V$ be a positive $m \times 1$ matrix, and let $W$ be a positive $1 \times n$ matrix such that the sum of the entries of $V$ equals the sum of the entries of $W$.
The top-left entry $x$ of the Kruithof limit of $A$ with target row sums $V$ and target column sums $W$ is algebraic over the field generated by the entries of $A$, $V$, and $W$, with degree at most $\binom{m + n - 2}{m - 1}$.
\end{conjecture}

Since the Kruithof limit specializes to the Sinkhorn limit when $V = \begin{bmatrix} 1 & 1 & \cdots & 1 \end{bmatrix}\tr$ and $W = \begin{bmatrix} \frac{m}{n} & \frac{m}{n} & \cdots & \frac{m}{n} \end{bmatrix}$, we suspect that the entries of $\adj_S(m, n)$ generalize in some way to involve entries of $V$ and $W$, and this should give a generalization of Conjecture~\ref{general polynomial} to Kruithof limits.
What are the generalized entries?

We mention that Conjecture~\ref{binomial sum} seems to generalize to Kruithof limits.
For example, the top-left entry $x$ of the Kruithof limit of a general positive $2 \times 3$ matrix $A$ with target row sums $V = \begin{bmatrix} r_1 & r_2 \end{bmatrix}\tr$ and target column sums $W = \begin{bmatrix} c_1 & c_2 & r_1 + r_2 - c_1 - c_2 \end{bmatrix}$ satisfies $f_3 x^3 + f_2 x^2 + f_1 x + f_0 = 0$, where
\begin{align*}
	f_3 &=
		M\!\left(\!\subsetmatrix{
			\{\} & \{2\}& \{2\} \\
			\{\} & \{2\}& \{3\}
		}\!\right) \\
	f_2 &=
		(c_2 - r_1 - r_2) M\!\left(\!\subsetmatrix{
			\{\} & \{2\} \\
			\{\} & \{2\}
		}\!\right)
		- (c_1 + c_2) M\!\left(\!\subsetmatrix{
			\{\} & \{2\} \\
			\{\} & \{3\}
		}\!\right)
		+ (r_2 - c_1) M\!\left(\!\subsetmatrix{
			\{2\} & \{2\} \\
			\{2\} & \{3\}
		}\!\right) \\
	f_1 &=
		c_1 (r_1 + r_2) M\!\left(\!\subsetmatrix{
			\{\} \\
			\{\}
		}\!\right)
		+ c_1 (r_1 - c_2) M\!\left(\!\subsetmatrix{
			\{2\} \\
			\{2\}
		}\!\right)
		+ c_1 (c_1 + c_2 - r_2) M\!\left(\!\subsetmatrix{
			\{2\} \\
			\{3\}
		}\!\right) \\
	f_0 &=
		-r_1 c_1^2 M\!\left(\!\subsetmatrix{
			\vphantom{\{} \\
			\vphantom{\{}
		}\!\right).
\end{align*}
The coefficients of the monomials $M(S)$ in the expressions for $f_k$ satisfy
\begin{multline*}
	x^3
	+ \big(
		(c_2 - r_1 - r_2)
		- (c_1 + c_2)
		+ (r_2 - c_1)
	\big) x^2 \\
	+ \big(
		c_1 (r_1 + r_2)
		+ c_1 (r_1 - c_2)
		+ c_1 (c_1 + c_2 - r_2)
	\big) x
	-r_1 c_1^2 \\
	= (x - r_1) (x - c_1)^2.
\end{multline*}
Additional Gr\"obner basis computations suggest the following.

\begin{conjecture}
With the notation of Conjecture~\ref{Kruithof limit degree}, write $V = \begin{bmatrix} r_1 & r_2 & \cdots & r_m \end{bmatrix}\tr$ and $W = \begin{bmatrix} c_1 & c_2 & \cdots & c_n \end{bmatrix}$ where $r_1 + r_2 + \dots + r_m = c_1 + c_2 + \dots + c_n$.
If the general equation satisfied by the top-left entry $x$ of the Kruithof limit of $A$ with target row sums $V$ and target column sums $W$ is
\[
	\sum_{S \subseteq D(m, n)} c_S(V, W) M(S) x^{\size{S}} = 0,
\]
where the equation is scaled so that $c_{D(m, n)}(V, W) = 1$, then
\[
	\sum_{S \subseteq D(m, n)} c_S(V, W) x^{\size{S}}
	= \left(x - r_1\right)^{\binom{m + n - 3}{n - 1}} \left(x - c_1\right)^{\binom{m + n - 3}{m - 1}}.
\]
In particular, this sum is independent of $r_2, \dots, r_m$ and $c_2, \dots, c_n$.
\end{conjecture}

\end{document}